\documentclass[a4paper]{amsart}

\pdfoutput=1

\usepackage{amsmath,amsthm,amssymb}
\usepackage{mathtools}
\usepackage{mathrsfs}
\usepackage{commath}

\usepackage{paralist}

\usepackage{microtype}
\usepackage{lmodern}
\usepackage[T1]{fontenc}
\usepackage[all]{xy}
\usepackage{dsfont}
\usepackage{chngcntr}
\usepackage[english]{babel}

\usepackage[autostyle]{csquotes}
\usepackage{multirow}
\usepackage{enumerate}
\usepackage{enumitem}

\usepackage{booktabs}
\usepackage{lunadiagrams}

\usepackage{varwidth}

\usepackage{tikz,pgffor}

\newtheorem{theorem}{Theorem}[section]
\newtheorem{lemma}[theorem]{Lemma}
\newtheorem{conj}[theorem]{Conjecture}
\newtheorem{prop}[theorem]{Proposition}

\newtheorem{cor}[theorem]{Corollary}
\theoremstyle{definition}
\newtheorem{definition}[theorem]{Definition}
\newtheorem{remark}[theorem]{Remark}
\newtheorem{example}[theorem]{Example}

\numberwithin{equation}{section}

\DeclareMathOperator*{\Supp}{supp}
\DeclareMathOperator*{\el}{e}
\DeclareMathOperator*{\vel}{v}
\DeclareMathOperator*{\reg}{reg}

\DeclareMathOperator*{\topint}{int}

\DeclareMathOperator*{\Spec}{Spec}
\DeclareMathOperator*{\Cl}{Cl}

\DeclareMathOperator{\rank}{rank}

\DeclareMathOperator{\lspan}{span}

\newcommand{\Sp}{\mathrm{Sp}}
\newcommand{\SO}{\mathrm{SO}}
\newcommand{\GL}{\mathrm{GL}}
\newcommand{\SL}{\mathrm{SL}}
\newcommand{\PSL}{\mathrm{PSL}}
\newcommand{\Hom}{\mathrm{Hom}}
\newcommand{\conv}{\mathrm{conv}}
\newcommand{\cone}{\mathrm{cone}}
\newcommand{\szs}{\Lambda}

\newcommand{\rleft}{\mathopen{}\mathclose\bgroup\left}
\newcommand{\rright}{\aftergroup\egroup\right}

\newcommand{\is}[1]{\wp(#1)}
\newcommand{\tc}[1]{\Rs_{#1}}
\newcommand{\adiv}{\Delta}
\newcommand{\bdiv}{\Dm}
\newcommand{\gdiv}{\Gamma}
\newcommand{\sv}{\vartheta}

\newcommand{\cur}{\Cs}

\newcommand{\Cd}{\mathds{C}}
\newcommand{\Qd}{\mathds{Q}}

\newcommand{\Zd}{\mathds{Z}}
\newcommand{\Rd}{\mathds{R}}

\newcommand{\Pd}{\mathds{P}}

\newcommand{\Td}{\mathds{T}}
\newcommand{\Ad}{\mathds{A}}

\newcommand{\C}{\Cd}
\newcommand{\R}{\Rd}
\newcommand{\Q}{\Qd}

\newcommand{\Z}{\Zd}

\newcommand{\Vm}{\mathcal{V}}
\newcommand{\Cm}{\mathcal{C}}
\newcommand{\Em}{\mathcal{E}}
\newcommand{\Fm}{\mathcal{F}}
\newcommand{\Dm}{\mathcal{D}}
\newcommand{\Rm}{\mathcal{R}}

\newcommand{\Om}{\mathcal{O}}
\newcommand{\Pm}{\mathcal{P}}
\newcommand{\Nm}{\mathcal{N}}
\newcommand{\Mm}{\mathcal{M}}

\newcommand{\Qm}{\mathcal{Q}}

\newcommand{\Xf}{\mathfrak{X}}

\newcommand{\Ss}{\mathscr{S}}
\newcommand{\Rs}{\mathscr{R}}
\newcommand{\Cs}{\mathscr{C}}

\newcommand{\ie}{i.\,e.~}
\newcommand{\eg}{e.\,g.~}

\newsavebox{\tinytwo}
\savebox{\tinytwo}{\put(-150,420){\tiny2}}
\newsavebox{\tinyone}
\savebox{\tinyone}{\put(-150,420){\tiny1}}
\newsavebox{\tinymone}
\savebox{\tinymone}{\put(-300,420){\tiny-1}}
\newsavebox{\tinyzero}
\savebox{\tinyzero}{\put(-150,420){\tiny0}}

\newsavebox{\dcircle}

\savebox{\dcircle}{
\put(300.000, 0.000){\circle*{0}}
\put(242.705, 176.336){\circle*{0}}
\put(92.705, 285.317){\circle*{0}}
\put(-92.705, 285.317){\circle*{0}}
\put(-242.705, 176.336){\circle*{0}}
\put(-300.000, 0.000){\circle*{0}}
\put(-242.705, -176.336){\circle*{0}}
\put(-92.705, -285.317){\circle*{0}}
\put(92.705, -285.317){\circle*{0}}
\put(242.705, -176.336){\circle*{0}}
}

\newsavebox{\gsecondtwox}
\savebox{\gsecondtwox}
{\put(0,0){\usebox{\lefttriedge}}\put(1800,0){\usebox{\gcircle}}}

\begin{document}

\title[The generalized Mukai conjecture for symmetric varieties]
{The generalized Mukai conjecture\\for symmetric varieties}

\author{Giuliano Gagliardi}
\address{Fachbereich Mathematik, Universit\"at T\"ubingen, Auf der
Morgenstelle 10, 72076 T\"ubingen, Germany}
\curraddr{Institut f\"ur Algebra, Zahlentheorie und Diskrete Mathematik, Leibniz Universit\"at Hannover,
Welfengarten 1, 30167 Hannover, Germany}
\email{gagliardi@math.uni-hannover.de}
\thanks{}

\author{Johannes Hofscheier}
\address{Fachbereich Mathematik, Universit\"at T\"ubingen, Auf der
Morgenstelle 10, 72076 T\"ubingen, Germany}
\curraddr{Institut f\"ur Algebra und Geometrie,
  Otto-von-Guericke-Universit\"at Magdeburg, Universit\"atsplatz 2,
  39106 Magdeburg, Germany}
\email{johannes.hofscheier@ovgu.de}
\thanks{}


\begin{abstract}
  We associate to any complete spherical variety $X$ a certain
  nonnegative rational number $\is{X}$, which we conjecture to satisfy
  the inequality $\is{X} \le \dim X - \rank X$ with equality holding
  if and only if $X$ is isomorphic to a toric variety.  We show that,
  for spherical varieties, our conjecture implies the generalized
  Mukai conjecture on the pseudo-index of smooth Fano varieties due to
  Bonavero, Casagrande, Debarre, and Druel.  We also deduce from our
  conjecture a smoothness criterion for spherical varieties.  It
  follows from the work of Pasquier that our conjecture holds for
  horospherical varieties.  We are able to prove our conjecture for
  symmetric varieties.
\end{abstract}

\maketitle

\microtypesetup{protrusion=false}
\tableofcontents
\microtypesetup{protrusion=true}

\section{Introduction}
\label{sec:introduction}

Let $X$ be a complex Gorenstein Fano variety, \ie a projective complex
algebraic variety with Cartier and ample anticanonical divisor
class. Recall the following generalization of a conjecture by Mukai
due to Bonavero, Casagrande, Debarre, and Druel. This conjecture
involves the \emph{pseudo-index} of $X$, \ie
\begin{align*}
  \iota_X \coloneqq \min{} \rleft\{ (-K_X \cdot C) : \text{$C$ is a
    rational curve in $X$}\rright\}\text{,}
\end{align*}
and the \emph{Picard number} of $X$, which we denote by $\rho_X$.

\begin{conj}[{\cite{bcdd}}]
  \label{conj:bcdd}
  Let $X$ be a smooth Fano variety. Then we have
  \begin{align*}
    \rho_X(\iota_X-1) \le \dim X\text{,}
  \end{align*}
  where equality holds if and only if $X \cong
  (\Pd^{\iota_X-1})^{\rho_X}$.
\end{conj}

In this paper, we investigate this conjecture in
the case of a \emph{spherical variety}
$X$, \ie a normal irreducible $G$-variety $X$ containing an open orbit
for a Borel subgroup $B$ of a connected reductive algebraic group
$G$. Spherical varieties can be considered as a generalization of
toric and horospherical varieties, for which
Conjecture~\ref{conj:bcdd} has already been proven by Casagrande
(see~\cite{cas06}) and Pasquier (see~\cite{pas10}) respectively.

To the spherical variety $X$ (not necessarily Gorenstein Fano) one may
assign the following combinatorial invariants: We denote by $\Mm$ the
weight lattice of $B$-semi-invariants in the function field $\C(X)$
and by $\adiv$ the set of $B$-invariant prime divisors in $X$.  To any
$D \in \adiv$ we associate the element $\rho'(D) \in \Nm \coloneqq
\Hom(\Mm, \Z)$ defined by $\langle \rho'(D), \chi\rangle \coloneqq
\nu_D(f_\chi)$ where $\langle \cdot, \cdot \rangle\colon \Nm \times \Mm \to
\Z$ denotes the natural pairing and $f_\chi \in \C(X)$ denotes a
$B$-semi-invariant rational function of weight $\chi \in \Mm$.  We
denote by $\rank X$ the rank of the lattice $\Mm$.

Brion has shown (see~\cite[Proposition~4.1]{Brion:cc}) that there is a natural choice of positive integers
$m_D$ such that
\begin{align*}
  -K_X \coloneqq \sum_{D\in\adiv} m_DD
\end{align*}
is an anticanonical divisor of $X$. Consider the intersection $Q^*$ of
half-spaces in the vector space $\Mm_\Q \coloneqq \Mm \otimes_\Z \Q$
defined as
\begin{align*}
  Q^* \coloneqq \bigcap_{D\in\adiv} \{v \in \Mm_\Q : \langle \rho'(D),
  v\rangle \ge -m_D \}\text{.}
\end{align*}
If $X$ is Gorenstein Fano, then $Q^*$ is a polytope (also known as
the moment polytope of $-K_X$) and certain of its vertices are called
\emph{supported} (they are in bijective correspondence with the closed
$G$-orbits in $X$).  The set of supported vertices of $Q^*$ is denoted
by $V_{\Supp}(Q^*)$. In fact, a vertex $v \in Q^*$ is supported if and
only if $v$ is the only point in the intersection of $Q^*$ with the
affine cone $v + \cone(\Sigma)$ where $\Sigma \subseteq \Mm_\Q$ is a
certain finite set associated to $X$ called the set of
\emph{spherically closed spherical roots}.  For details, we refer the
reader to Section~\ref{sec:spher-syst-luna}.

In the case of a horospherical variety $X$ (\ie $\Sigma = \emptyset$),
the inequality of Conjecture~\ref{conj:bcdd} follows from the
inequality $\sum_{D \in\adiv} (m_D - 1) \le \dim X - \rank X$
(see~\cite[Proof of the inequality of Theorem~1]{pas10}).  By refining
this approach of Pasquier, we obtain the following generalization.

\begin{prop}
  \label{prop:mineq}
  Let $X$ be a $\Q$-factorial Gorenstein spherical Fano variety.
  Assume that there exists $\sv \in \conv(V_{\Supp}(Q^*))$ such that
  \begin{align*}
    \sum_{D \in \adiv} \rleft(m_D - 1 + \langle \rho'(D), \sv
    \rangle\rright) \le \dim X - \rank X \text{.}
  \end{align*}
  Then we have $\rho_X(\iota_X-1) \le \dim X$.
\end{prop}

Motivated by Proposition~\ref{prop:mineq} and the fact that the
intersection $\conv(V_{\Supp}(Q^*)) \cap \cone(\Sigma)$ is not empty
(see~\cite[Lemma~13.3]{gsfv}), we propose to study the following
invariant in order to handle the non-horospherical case (\ie $\Sigma
\ne \emptyset$).

\begin{definition}
  \label{def:is}
  For an arbitrary spherical variety $X$, we define
  \begin{align*}
    \is{X} \coloneqq \sup{} \rleft\{ \sum_{D \in \adiv} \rleft( m_D -
    1 + \langle \rho'(D), \sv \rangle \rright): \sv \in Q^* \cap
    \cone(\Sigma)\rright\} \in \Q_{\ge0} \cup \{\infty\}\text{.}
  \end{align*}
  We will explain in Section~\ref{sec:spherical-skeletons} why the
  value $\is{X}$ is nonnegative and rational.
\end{definition}

We can now state our main conjecture as well as its main implication.

\begin{conj}
  \label{conj:esgp}
  Let $X$ be a complete spherical variety. Then we have
  \begin{align*}
    \is{X} \le \dim X - \rank X\text{,}
  \end{align*}
  where equality holds if and only if $X$ is isomorphic to a toric
  variety.
\end{conj}

\begin{theorem}
  \label{th:bcddsph}
  Let $X$ be a $\Q$-factorial Gorenstein spherical Fano variety and
  assume that Conjecture~\ref{conj:esgp} holds for $X$. Then we have
  \begin{align*}
    \rho_X(\iota_X-1) \le \dim X\text{,}
  \end{align*}
  where equality holds if and only if $X \cong
  (\Pd^{\iota_X-1})^{\rho_X}$.
\end{theorem}

Our conjecture also implies a smoothness criterion for spherical
varieties. The local structure theorem for spherical varieties always
reduces the question of smoothness to the case where $X$ is affine and
the derived subgroup $[G, G]$ fixes pointwise the unique closed
$G$-orbit in $X$.

\begin{theorem}
  \label{th:smoothness}
  Let $X$ be a locally factorial affine spherical variety such that
  the derived subgroup $[G,G]$ fixes pointwise the unique closed
  $G$-orbit in $X$.  If Conjecture~\ref{conj:esgp} holds, then $X$ is
  smooth if and only if $\is{X} = \dim X - \rank X$.
\end{theorem}

A comparison with the smoothness criterion
given in \cite{gcscsv} shows that the condition $\is{X} = \dim X - \rank X$
in Theorem~\ref{th:smoothness} would replace the otherwise
necessary consultation of the list in \cite[Section~2]{gcscsv}.

\begin{remark}
  It follows from the work of Pasquier that Conjecture~\ref{conj:esgp}
  holds for horospherical varieties (see
  \cite[Lemme~4.8]{Pasquier:FanoHorospherical} and
  \cite[Lemma~8(iv)]{pas10}).
\end{remark}

We are able to prove Conjecture~\ref{conj:esgp} in the following
further situation: A normal irreducible $G$-variety $X$ is called
\emph{symmetric} if there exists a nontrivial involution $\theta\colon
G \to G$ such that $X$ contains an open $G$-orbit isomorphic to $G/H$
with $(G^{\theta})^\circ \subseteq H \subseteq N_G(G^{\theta})$ where
$G^{\theta} \subseteq G$ denotes the set of fixed points under
$\theta$. Symmetric varieties are known to be spherical (see, for instance,
\cite[Theorem~26.14]{ti}).

\begin{theorem}
  \label{th:sym}
  Conjecture~\ref{conj:esgp} holds for symmetric varieties.
\end{theorem}

\subsection*{List of general notation}
\renewcommand{\descriptionlabel}[1]{\hspace\labelsep #1}
\begin{description}[leftmargin=8em,style=nextline]
\item[{$\Pm(A)$}] power set of a set $A$,
\item[{$G^\circ$}] identity component of an algebraic group $G$,
\item[{$\Lambda_\Q$}] $\Lambda \otimes_\Z \Q$ for a lattice $\Lambda$,
\item[{$\Cm^\vee$}] dual cone to a cone $\Cm$ in a vector space $V$, \ie \\
  $\Cm^\vee \coloneqq \{ v \in V^*: \langle u, v \rangle \ge 0 \text{
    for every $u \in \Cm$}\}$,
\item[{$Q^*$}] dual polytope to a polytope $Q$ in a vector space $V$, \ie \\
  $Q^* \coloneqq \{ v \in V^* : \langle u, v \rangle \ge -1 \text{ for
    every $u \in Q$}\}$,
\item[{$\widehat{F}$}] dual face to a face $F$ of a polytope $Q$, \ie \\
  $\widehat{F} \coloneqq \{ v \in Q^* : \langle u, v \rangle = -1
  \text{ for every $u \in F$}\}$,
\item[{$\topint(A)$}] topological interior of a subset $A$ in some
  finite-dimensional vector space,
\item[{$A^\circ$}] relative interior of a subset $A$ in some finite-dimensional
vector space. 
\end{description}

\section{Spherical systems and Luna diagrams}
\label{sec:spher-syst-luna}

Spherical varieties can be described combinatorially.  A closed
subgroup $H$ of a connected reductive complex algebraic group $G$ is
called \emph{spherical} if $G/H$ contains an open orbit for a Borel
subgroup, and then $G/H$ is called a \emph{spherical homogeneous
  space}.  The Luna conjecture provides a description of the spherical
subgroups of a fixed connected reductive group
(see~\cite{Luna:typea}), which has recently been proven (see~\cite{bp,
  cf2, losev-uniq}).  Then, for a fixed spherical homogeneous space
$G/H$, the historically earlier Luna-Vust theory (see~\cite{lunavust,
  knopsph}) describes the $G$-equivariant open embeddings $G/H
\hookrightarrow X$ into a normal irreducible $G$-variety $X$.

In this section, we give a brief summary on the combinatorial
description of spherical subgroups, \ie the first of the two steps
mentioned above.  This section does not contain any new results.  In
addition to the references above, we use \cite{f4},
\cite[Section~30.11]{ti}, and \cite[Section~2]{BS:Moduli} as general
references.

\subsection*{Spherically closed spherical systems}
Let $R$ be a root system, and let $S \subseteq R$ be a choice of
simple roots. We write $\Xf(R)$ for the root lattice of $R$.  For
$\gamma \in \Xf(R)$ we write $\Supp(\gamma) \subseteq S$ for the
support of $\gamma$, \ie the set of simple roots having a nonzero
coefficient in the expression for $\gamma$ as linear combination of
simple roots.

\begin{definition}
  The set $\Sigma^{sc}(R)$ of \emph{spherically closed spherical roots
    of $R$} is defined to consist of those elements $\gamma \in
  \Xf(R)$ which are listed in the second column of Table~\ref{tab:sr}
  after applying the usual Bourbaki numbering
  (see~\cite{Bourbaki:Lie}) to the simple roots in $\Supp(\gamma)$.
\end{definition}

\begin{table}[!ht]
  \begin{tabular}{ccc}
    \toprule
    diagram & spherical root & coefficient \\
    \midrule
    \parbox[c][][c]{4cm}{\centering
      \begin{picture}(0,2100)(0,-1050)
        \put(0,0){\usebox{\aone}}
        \put(0,0){\usebox{\tinyone}}
        \put(0,-1200){\usebox{\tinyone}}
      \end{picture}}
    & \parbox[c][][c]{4cm}{\centering
      $\alpha_1$} & $1$ \\
    \parbox[c][][c]{4cm}{\centering
      \begin{picture}(0,2100)(0,-1050)
        \put(0,0){\usebox{\aprime}}
        \put(0,-1200){\usebox{\tinytwo}}
      \end{picture}}
    & \parbox[c][][c]{4cm}{\centering
      $2\alpha_1$} & $1$ \\
    \parbox[c][][c]{4cm}{\centering
      \begin{picture}(1800,2100)(0,-1050)
        \put(900,-1200){\usebox{\tinytwo}}
        \multiput(0,0)(1800,0){2}{\usebox{\vertex}}
        \multiput(0,0)(1800,0){2}{\usebox{\wcircle}}
        \multiput(0,-300)(1800,0){2}{\line(0,-1){600}}
        \put(0,-900){\line(1,0){1800}}
      \end{picture}}
    & \parbox[c][][c]{4cm}{\centering
      $\alpha_1 + \alpha'_1$} & $2$ \\
    \parbox[c][][c]{4cm}{\centering
      \begin{picture}(5400,2100)(0,-1050)
        \put(0,0){\usebox{\mediumam}}
        \put(0,0){\usebox{\tinyone}}
        \put(5400,0){\usebox{\tinyone}}
      \end{picture}}
    & \parbox[c][][c]{4cm}{\centering
      $\alpha_1 + \dots + \alpha_n$} & $n$ \\
    \parbox[c][][c]{4cm}{\centering
      \begin{picture}(3600,2100)(0,-1050)
        \put(0,0){\usebox{\dthree}}
        \put(1800,0){\usebox{\tinytwo}}
      \end{picture}}
    & \parbox[c][][c]{4cm}{\centering
      $\alpha_1 + 2\alpha_2 + \alpha_3$} & $4$ \\
    \parbox[c][][c]{4cm}{\centering
      \begin{picture}(7200,2100)(0,-1050)
        \put(7200,0){\usebox{\dcircle}}
        \put(0,0){\usebox{\shortbm}}
	\put(0,0){\usebox{\tinyone}}
      \end{picture}}
    & \parbox[c][][c]{4cm}{\centering
      $\alpha_1 + \dots + \alpha_n$} & $n$ \\
    \parbox[c][][c]{4cm}{\centering
      \begin{picture}(7200,2100)(0,-1050)
        \put(0,0){\usebox{\shortbprimem}}
      \end{picture}}
    & \parbox[c][][c]{4cm}{\centering
      $2\alpha_1 + \dots + 2\alpha_n$} & $2n - 1$ \\
    \parbox[c][][c]{4cm}{\centering
      \begin{picture}(3600,2100)(0,-1050)
        \put(3600,0){\usebox{\tinytwo}}
        \put(0,0){\usebox{\bthirdthree}}
      \end{picture}}
    & \parbox[c][][c]{4cm}{\centering
      $\alpha_1 + 2\alpha_2 + 3\alpha_3$} & $6$ \\
    \parbox[c][][c]{4cm}{\centering
      \begin{picture}(9000,2100)(0,-1050)
        \put(0,0){\usebox{\dcircle}}
        \put(0,0){\usebox{\shortcm}}
        \put(1800,0){\usebox{\tinyone}}
      \end{picture}}
    & \parbox[c][][c]{4cm}{\centering
      $\alpha_1 + 2\alpha_2 + \dots + 2\alpha_{n-1} + \alpha_n$} & $2n - 2$ \\
    \parbox[c][][c]{4cm}{\centering
      \begin{picture}(9000,2100)(0,-1050)
        \put(0,0){\usebox{\shortcm}}
        \put(1800,0){\usebox{\tinyone}}
      \end{picture}}
    & \parbox[c][][c]{4cm}{\centering
      $\alpha_1 + 2\alpha_2 + \dots + 2\alpha_{n-1} + \alpha_n$} & $2n -1$ \\
    \parbox[c][][c]{4.5cm}{\centering
      \begin{picture}(6600,2100)(0,-1050)
        \put(0,0){\usebox{\tinytwo}}
        \put(0,0){\usebox{\shortdm}}
      \end{picture}}
    & \parbox[c][][c]{4.5cm}{\centering
      $2\alpha_1 + \dots + 2\alpha_{n-2} + \alpha_{n-1} + \alpha_n$} & $2n-2$ \\
    \parbox[c][][c]{4cm}{\centering
      \begin{picture}(5400,2100)(0,-1050)
        \put(0,0){\usebox{\ffour}}
        \put(5400,0){\usebox{\tinyone}}
      \end{picture}}
    & \parbox[c][][c]{4cm}{\centering
      $\alpha_1 + 2\alpha_2 + 3\alpha_3 + 2\alpha_4$} & $11$ \\
    \parbox[c][][c]{4cm}{\centering
      \begin{picture}(1800,2100)(0,-1050)
        \put(0,0){\usebox{\gsecondtwox}}
        \put(0,0){\usebox{\dcircle}}
        \put(1800,0){\usebox{\tinyone}}
      \end{picture}}
    & \parbox[c][][c]{4cm}{\centering
      $\alpha_1 + \alpha_2$} & $2$ \\
    \parbox[c][][c]{4cm}{\centering
      \begin{picture}(1800,2100)(0,-1050)
        \put(0,0){\usebox{\gprimetwo}}
      \end{picture}}
    & \parbox[c][][c]{4cm}{\centering
      $4\alpha_1 + 2\alpha_2$} & $5$ \\
    \bottomrule
  \end{tabular}
  \caption{Spherically closed spherical roots.}
  \label{tab:sr}
\end{table}

If $\Sigma \subseteq \Sigma^{sc}(R)$ is a subset, we write $\Lambda
\subseteq \Xf(R)$ for the sublattice spanned by $\Sigma$ and
$\Lambda^* \coloneqq \Hom(\Lambda,\Z)$ for its dual lattice.

\begin{definition}
  \label{def:adapted}
  Let $\Sigma \subseteq \Sigma^{sc}(R)$ be a subset, and let $\Dm^a$
  be a multiset consisting of (not necessarily distinct) elements of
  $\Lambda^*$.  It is convenient to consider $\Dm^a$ as an abstract
  set equipped with a map $\rho \colon \Dm^a \to \Lambda^*$.  For
  every $\alpha \in \Sigma \cap S$ we define $\Dm^a(\alpha)
  \coloneqq\{ D \in \Dm^a : \langle \rho(D), \alpha \rangle = 1\}$. We
  say that $\Dm^a$ is \emph{adapted to $\Sigma$} if the following
  properties are satisfied:
  \begin{enumerate}
  \item[($\operatorname{A1}$)] For every $D \in \Dm^a$, $\gamma \in
    \Sigma$ we have $\langle \rho(D), \gamma \rangle \le 1$ with
    equality holding if and only if $\alpha \coloneqq \gamma \in
    \Sigma \cap S$ and $D \in \Dm^a(\alpha)$.
  \item[($\operatorname{A2}$)] For every $\alpha \in \Sigma \cap S$
    the set $\Dm^a(\alpha)$ consists of exactly two elements
    $D_\alpha^+$, $D_\alpha^-$ such that $\rho(D_\alpha^+) +
    \rho(D_\alpha^-) = \alpha^\vee|_\Lambda$.
  \item[($\operatorname{A3}$)] We have $\Dm^a = \bigcup_{\alpha \in
      \Sigma \cap S} \Dm^a(\alpha)$.
  \end{enumerate}
\end{definition}

\begin{definition}
  A \emph{spherically closed spherical $R$-system} is a triple
  $(\Sigma, S^p, \Dm^a)$ consisting of $\Sigma \subseteq
  \Sigma^{sc}(R)$, $S^p \subseteq S$, and $\Dm^a$ adapted to $\Sigma$
  such that the following axioms are satisfied:
  \begin{enumerate}
  \item[($\operatorname{\Sigma1}$)] If $\alpha \in (\tfrac{1}{2}\Sigma)
    \cap S$, then $\langle \alpha^\vee, \szs\rangle \subseteq 2\Z$ and
    $\langle \alpha^\vee, \Sigma \setminus \{2\alpha\}\rangle \le 0$.
  \item[($\operatorname{\Sigma2}$)] If $\alpha, \beta \in S$ are
    orthogonal and $\alpha + \beta \in \Sigma$, then
    $\alpha^\vee|_\szs = \beta^\vee|_\szs$.
  \item[($\operatorname{S}$)] Every $\gamma \in \Sigma$ is
    \emph{compatible} with $S^p$, \ie
    \begin{enumerate}
    \item[1)] the set $\Supp(\gamma) \cap S^p$ coincides with the set
      of vertices without any (normal, shadowed, or dotted) circles
      around, above, or below themselves in the first column of
      Table~\ref{tab:sr},
    \item[2)] the simple roots in $S^p \setminus \Supp(\gamma)$ are
      orthogonal to $\gamma$.
    \end{enumerate}
  \end{enumerate}
\end{definition}

\begin{remark} 
  Our definition of \enquote{compatible} (which is equivalent to the
  definition in \cite[Definition~2.5]{BS:Moduli}) is more restrictive
  than the definition in \cite[1.1.6]{f4}, where a more general class
  of not necessarily spherically closed spherical systems is
  considered.
\end{remark}

\begin{remark}
  An inspection of Table~\ref{tab:sr} shows that in the situation of
  axiom~($\operatorname{S}$) we have in fact $\langle \alpha^\vee,
  \gamma\rangle = 0$ for every $\alpha \in S^p$ (not only for $\alpha
  \in S^p \setminus \Supp \gamma$).
\end{remark}

\begin{definition}
  If $\Ss_1 \coloneqq (\Sigma_1, S^p_1, \Dm^a_1)$ and $\Ss_2 \coloneqq
  (\Sigma_2, S^p_2, \Dm^a_2)$ are spherically closed spherical systems
  for root systems $R_1$ and $R_2$ respectively, we define the product
\begin{align*}
  \Ss_1 \times \Ss_2 \coloneqq (\Sigma_1\cup\Sigma_2, S^p_1 \cup
  S^p_2, \Dm^a_1 \cup \Dm^a_2)
\end{align*}
with $\langle \rho(\Dm_1^a), \Sigma_2 \rangle = \{0\}$ and $\langle
\rho(\Dm_2^a), \Sigma_1 \rangle = \{0\}$, which is a spherically
closed spherical system for the root system $R_1 \times R_2$.
\end{definition}

\subsection*{Augmentations of spherically closed spherical systems}
Let $G$ be a connected reductive complex algebraic group, let $B
\subseteq G$ be a Borel subgroup, and let $T \subseteq B$ be a maximal
torus such that $(G, T)$ has root system $R$ and $S$ is the set of
simple roots corresponding to $B$.

\begin{definition}
  Let $\Ss \coloneqq (\Sigma, S^p, \Dm^a)$ be a spherically closed
  spherical $R$-system. An \emph{augmentation of $\Ss$ for $G$} is a
  pair $(\Mm, \rho')$ consisting of a sublattice $\Mm \subseteq
  \Xf(B)$ containing $\Sigma$ and a map $\rho' \colon \Dm^a \to \Nm
  \coloneqq \Hom( \Mm, \Z )$ such that the following axioms are
  satisfied:
  \begin{enumerate}
  \item[($\operatorname{a1}$)] For every $D \in \Dm^a$ we have
    $\rho'(D)|_\Lambda = \rho(D)$.
  \item[($\operatorname{a2}$)] For every $\alpha \in \Sigma \cap S$ we
    have $\rho'(D_\alpha^+) + \rho'(D_\alpha^-) = \alpha^\vee|_{\Mm}$.
  \item[($\operatorname{\sigma 1}$)] If $\alpha \in (\tfrac{1}{2}
    \Sigma) \cap S$, then $\alpha \not \in \Mm$ and $\langle
    \alpha^\vee, \Mm \rangle \subseteq 2\Z$.
  \item[($\operatorname{\sigma 2}$)] If $\alpha, \beta \in S$ are
    orthogonal and $\alpha + \beta \in \Sigma$, then
    $\alpha^\vee|_{\Mm} = \beta^\vee|_{\Mm}$.
  \item[($\operatorname{s}$)] For every $\alpha \in S^p$ we have
    $\langle \alpha^\vee, \Mm \rangle = \{ 0 \}$.
  \end{enumerate}
  A spherically closed spherical system together with an augmentation
  is called an \emph{augmented spherically closed spherical system}.
\end{definition}

We explain how to associate to a spherical subgroup $H \subseteq G$ an
augmented spherically closed spherical system.  The stabilizer of the
open $B$-orbit in $G/H$ is a parabolic subgroup of $G$ containing $B$,
hence uniquely determines a set $S^p \subseteq S$ of simple roots.

We denote by $\Mm \subseteq \Xf(B)$ the weight lattice of
$B$-semi-invariants in the function field $\C(G/H)$ and by $\bdiv$ the
set of $B$-invariant prime divisors in $G/H$.  The elements of $\bdiv$
are called the \emph{colors}. To any $D \in \bdiv$ we associate the
element $\rho'(D) \in \Nm \coloneqq \Hom(\Mm, \Z)$ defined by $\langle
\rho'(D), \chi\rangle \coloneqq \nu_D(f_\chi)$ where $\langle \cdot,
\cdot \rangle\colon \Nm \times \Mm \to \Z$ denotes the natural pairing and
$f_\chi \in \C(G/H)$ denotes a $B$-semi-invariant rational function of
weight $\chi \in \Mm$ (which is uniquely determined up to 
a constant factor). For $\alpha \in S$ we write $\Dm(\alpha)
\subseteq \Dm$ for the subset of colors moved by the minimal parabolic
subgroup $P_\alpha \subseteq G$ containing $B$ and
corresponding to the simple root
$\alpha$, \ie the colors $D$ such that $P_\alpha \cdot D \ne D$.

We denote by $\Vm$ the set of $G$-invariant discrete valuations $\nu\colon
\C(G/H)^* \to \Q$.  The assignment $\langle \nu, \chi \rangle
\coloneqq \nu(f_\chi)$ for $\nu \in \Vm$ induces an inclusion $\Vm
\subseteq \Nm_\Q$, where $\Vm$ is known to be a cosimplicial cone
(see~\cite{brg}) called the \emph{valuation cone}. Every extremal ray
of the simplicial cone $-\Vm^\vee \subseteq \Mm_\Q$ contains at least
one spherically closed spherical root which is compatible with
$S^p$. We pick the shortest one of these from every extremal ray of
$-\Vm^\vee$ and denote by $\Sigma$ their collection.

Finally, we set $\Dm^a \coloneqq \bigcup_{\alpha \in \Sigma \cap S}
\Dm(\alpha)$ and define $\rho\colon \Dm \to \Lambda^*$ by $\rho(D)
\coloneqq \rho'(D)|_{\Lambda}$.  It will not be harmful to drop the
notation of restricting to $\Dm^a \subseteq \Dm$, \ie we simply write
$\rho'$ for $\rho'|_{\Dm^a}$ and $\rho$ for $\rho|_{\Dm^a}$.

\begin{prop}
  \label{prop:clc}
  For a spherical subgroup $H \subseteq G$ the triple $(\Sigma, S^p,
  \Dm^a)$ defined as in the preceding paragraph is a spherically
  closed spherical $R$-system, and $(\Mm, \rho')$ is an augmentation
  for $G$.  This assignment defines a bijection
  \begin{align*}
    \rleft\{\begin{varwidth}{5in}{spherical subgroups $H \subseteq
        G$\\up to conjugation}
    \end{varwidth}\rright\}
    \leftrightarrow \rleft\{\begin{varwidth}{5in}{augmented
        spherically closed\\spherical systems for
        $G$}\end{varwidth}\rright\}\text{.}
  \end{align*}
\end{prop}
\begin{proof}
This is \cite[Proposition~6.4]{Luna:typea} together with
the Luna conjecture mentioned at the beginning of the present
section, which states that spherically closed spherical
subgroups correspond to spherically closed spherical systems.
\end{proof}

\begin{remark}
  Let $H \subseteq G$ be a spherical subgroup, and let $\Ss \coloneqq
  (\Sigma, S^p, \Dm^a)$ be the associated spherically closed spherical
  system. The \emph{trivial} pair $(\Lambda, \rho)$ is an augmentation
  of $\Ss$ for $G$.  The corresponding spherical subgroup $\overline{H} \subseteq G$,
  which is called the \emph{spherical closure} of $H$,
  can be obtained as follows: We identify the $G$-equivariant
  automorphism group of $G/H$ with $N_G(H)/H$.  Then $N_G(H)$ acts on
  $\Dm$, and we define $\overline{H} \subseteq N_G(H)$ to be the
  kernel of this action. 
\end{remark}

\subsection*{Luna diagrams of spherically closed spherical systems}

A \emph{Luna diagram} is a convenient way to visualize a spherically
closed spherical $R$-system $\Ss \coloneqq (\Sigma, S^p, \Dm^a)$.  We
recall from \cite[1.2.4]{f4} how a Luna diagram is arranged.

We begin by drawing the Dynkin diagram of the root system $R$.  The
spherically closed spherical roots in $\Sigma$ are represented as
indicated in the first column of Table~\ref{tab:sr}, but without the
dotted circles and without the small numbers.  Note that we deviate
from \cite[1.2.4]{f4}, where small numbers are sometimes required to
be drawn (for an explanation, see Remark~\ref{rem:scsr}).

We continue by representing the set $S^p$. This is done by drawing a
(non-shadowed) circle around each vertex corresponding to a simple
root not in $S^p$ which does not yet have any circle around, above, or
below itself.

It remains to visualize $\Dm^a$. Recall that for every $\alpha \in
\Sigma \cap S$ there are exactly two elements in $\Dm(\alpha)
\subseteq \Dm^a$, which we denote by $D_\alpha^+$ and $D_\alpha^-$.
It is always possible to choose $D_\alpha^+$ such that $\langle
\rho(D_\alpha^+), \Sigma \rangle \subseteq \{1, 0,-1\}$.  We identify
$D_\alpha^+$ with the circle above and $D_\alpha^-$ with the circle
below the vertex $\alpha$.  Let $\gamma \in \Sigma$ be distinct from
$\alpha$. We have $\langle \rho(D_\alpha^+), \gamma \rangle = 1$ if
and only if $\beta \coloneqq \gamma \in S$ and $D_{\alpha}^+ =
D_{\beta}^+$ or $D_{\alpha}^+ = D_{\beta}^-$.  Otherwise, if $\gamma$
is not orthogonal to $\alpha$, the value $\langle \rho(D_\alpha^+),
\gamma \rangle$ could be either $0$ or $-1$.  This ambiguity is
resolved by adding an arrow starting from the circle corresponding to
$D_\alpha^+$ pointing towards $\gamma$ in the case where $\langle
\rho(D_\alpha^+), \gamma \rangle = -1$.  In order to complete the Luna
diagram, we join by a line those circles corresponding to the same
element in $\Dm^a$. 

It is clear that a spherically closed spherical system may be
recovered from its Luna diagram: The set $\Sigma$ can be directly read
off the Luna diagram by inspecting Table~\ref{tab:sr}. The set
$S^p$ consists exactly of those simple roots which do not have a
circle around, above, or below themselves.
Now consider the set of those circles
which are above or below a
simple root having circles both above and below itself.
Circles in this set are considered to be equivalent if they are
joined by a line. The abstract set $\Dm^a$ is defined to coincide with the set
of equivalence classes of such circles.
It remains to determine the map $\rho \colon \Dm^a \to
\Lambda^*$ (where $\Lambda \subseteq \Xf( R )$ is the sublattice
generated by $\Sigma$). This is achieved by exploiting properties
(A1) and (A2) of Definition~\ref{def:adapted} and
the rule for placing the arrows as described in the preceding paragraph.

The full list of Luna diagrams for symmetric spaces is given in
Appendix~\ref{sec:luna-diagr-symm}.

\begin{example}
  Consider the Luna diagram No.~3 for $l = 1$, $m \ge 1$ in Appendix
  \ref{sec:luna-diagr-symm} with the usual numbering of simple
  roots. Then, according to Table \ref{tab:sr}, we have $\Sigma = \{ \alpha_k {+}
  \alpha_{2m{+}2{-}k} : k = 1, 2, \ldots, m \} \cup \{ \alpha_{m{+}1}
  \}$. As every simple root has a circle around, above, or below itself, we have $S^p
  = \emptyset$.

  It remains to determine the map $\rho \colon \Dm^a \to \Lambda^*$ for
  the abstract set $\Dm^a = \{ D_{\alpha_{m{+}1}}^{+}, D_{\alpha_{m{+}1}}^{-} \}$.
  By property (A1) of Definition~\ref{def:adapted},
  we have $\langle \rho(D_{\alpha_{m{+}1}}^{+}), \alpha_{m{+}1}\rangle =
  \langle \rho(D_{\alpha_{m{+}1}}^{-}), \alpha_{m{+}1}\rangle = 1$.
  The arrow means that we have
  $\langle \rho(D_{\alpha_{m{+}1}}^{+}), \alpha_{m}+\alpha_{m{+}2}\rangle = -1$.
  As we have $\langle \alpha_{m{+}1}^\vee, \alpha_{m}+\alpha_{m{+}2}\rangle = -2$,
  we obtain $\langle \rho(D_{\alpha_{m{+}1}}^{-}), \alpha_{m}+\alpha_{m{+}2}\rangle = -1$
  from property (A2) of Definition~\ref{def:adapted}. Finally, consider $\gamma \coloneqq
  \alpha_k + \alpha_{2m{+}2{-}k}$ for $1 \le k < m$. As we have
  $\langle \alpha^\vee_{m{+}1}, \gamma \rangle = 0$,
  we obtain $\langle \rho(D_{\alpha_{m{+}1}}^{+}), \gamma\rangle =
  \langle \rho(D_{\alpha_{m{+}1}}^{-}), \gamma\rangle = 0$ from properties (A1) and (A2)
  of Definition~\ref{def:adapted}. 
\end{example}

\subsection*{The full set of colors}
Let $\Ss$ be a spherically closed spherical system and $(\Mm, \rho')$
an augmentation. As always, we denote by $\Dm$ the set of colors of
its associated spherical homogeneous space, equipped with the map
$\rho'\colon \Dm \to \Nm$.

The augmented spherically closed spherical system only retains the
restricted map $\rho'\colon \Dm^a \to \Nm$, but the full map
$\rho'\colon \Dm \to \Nm$ may be recovered as follows
(for details, we refer to \cite[2.3]{Luna:typea}, see also \cite[1.2.4]{f4}): The elements of
$\Dm$ correspond to the circles in the Luna diagram of $\Ss$ where
elements are identified whenever the corresponding circles are joined
by a line. Under this correspondence, the set $\Dm(\alpha)$ of colors
moved by $P_\alpha$ contains exactly those colors which correspond to
a circle above, below, or around the vertex $\alpha$.

If $D \in \Dm$ corresponds to a circle above or below $\alpha \in \Sigma \cap S$,
we have $D \in \Dm^a$ and $\rho'(D)$ is known. If $D$
corresponds to a circle below $\alpha \in (\tfrac{1}{2}\Sigma) \cap S$,
we have $\rho'(D) = \frac{1}{2}\alpha^\vee|_{\Mm}$. Otherwise, $D$
corresponds to a circle around $\alpha \in S$, and we have $\rho'(D) =
\alpha^\vee|_{\Mm}$.

\subsection*{The canonical divisor of a spherical variety}

We consider an arbitrary spherical variety $X$ with open $G$-orbit
$G/H \subseteq X$. We denote by $\Delta$ the set of $B$-invariant
prime divisors in $X$.  By naturally identifying the $B$-invariant
prime divisors in $G/H$ with their closures in $X$, we obtain an
inclusion $\bdiv \subseteq \adiv$ and can extend the map $\rho' \colon
\bdiv \to \Nm$ to a map $\rho' \colon \adiv \to \Nm$ in the obvious
way. The notation in this section is then in agreement with the
introduction.

Brion has shown (see~\cite[Proposition~4.1]{Brion:cc}) that there is a
natural choice of positive integers $m_D$ such that
\begin{align*}
  -K_X \coloneqq \sum_{D\in\adiv} m_DD
\end{align*}
is an anticanonical divisor of $X$, where we have $m_D = 1$ for $D \in
\adiv \setminus \bdiv$ (\ie for the $G$-invariant prime divisors $D$).
Moreover, the coefficients $m_D$ for $D \in \bdiv$ depend only on the
spherical homogeneous space $G/H$.

In fact, they depend only on the spherically closed spherical system
associated to the spherical homogeneous space $G/H$, as can be seen
from the following formulae due to Luna (see \cite[Section
3.6]{Luna:cc}): For $I \subseteq S$ we denote by $\rho_I$ the half-sum
of positive roots in the root system generated by $I$.  Let $\alpha
\in S \setminus S^p$ and $D \in \Dm(\alpha)$.  Then we have
\begin{align*}
  m_D & = 1 &&\text{ if $\alpha \in \Sigma$ or $2\alpha \in \Sigma$,}\\
  m_D &= \langle \alpha^\vee, 2\rho_S - 2\rho_{S^p} \rangle && \text{
    otherwise.}
\end{align*}

\subsection*{Explanation of Table~\ref{tab:sr}}
For a spherically closed spherical root $\gamma$, the first column
shows how it is represented in Luna diagrams on the Dynkin diagram
corresponding to $\Supp \gamma$. The vertices without any (normal,
shadowed, or dotted) circles around, above, or below themselves
correspond to the set $\Supp \gamma \cap S^p$.
  
The normal and shadowed circles correspond to colors whose existence
is implied by the spherically closed spherical root $\gamma$.  If $D$
is such a color, then the small number near the corresponding circle
is the value of $\langle \rho(D), \gamma \rangle$.  On a Luna diagram,
the dotted circles will always be replaced by some other type of
circle (this happens automatically when a Luna diagram is arranged
according to the above description).
 
The spherical root $\alpha_1 + 2\alpha_2 + \dots + 2\alpha_{n-1} +
\alpha_n$ with support of type $C_n$ appears twice in the table, but
with the set $\Supp \gamma \cap S^p$ differing, which means that both
possibilities are allowed.

It follows from the above formulae and axiom $(\operatorname{S})$ that
the coefficients $m_D$ in the expression for the anticanonical divisor
are determined by $\gamma$ and $\Supp \gamma \cap S^p$ for the colors
$D$ corresponding to a normal or shadowed circle in the first
column. The coefficient $m_D$ is given in the third column. If more
than one color appears in the first column, the coefficient will apply
to both of them.

\begin{remark}
  \label{rem:scsr}
  We explain how Table~\ref{tab:sr} differs from
  \cite[Table~2]{f4}. The notions introduced here will not be used
  elsewhere.
  
  Let $(\Sigma, S^p, \Dm^a)$ be a spherically closed spherical system,
  and let $(\Mm, \rho')$ be an augmentation corresponding to the
  spherical subgroup $H \subseteq G$.  For $\gamma \in \Sigma$ we
  denote by $\gamma' \in \Mm$ the corresponding unique primitive
  element $\gamma' \in \Mm$ such that $\gamma = k\gamma'$ for some
  positive integer $k$ (in fact, we always have $k \in \{1, 2\}$).
  The elements of the set $\Sigma' \coloneqq \{\gamma' : \gamma \in
  \Sigma\}$ are called the \emph{spherical roots}.

  The spherical subgroup $H \subseteq G$ is called \emph{wonderful} if
  $\Mm = \lspan_{\Z} \Sigma'$. If one is interested in wonderful
  subgroups, instead of considering the augmented spherically closed
  spherical system with the additional condition $\Mm = \lspan_{\Z}
  \Sigma'$, it can be more convenient to just consider the triple
  $(\Sigma', S^p, \Dm^a)$, which is called a \emph{spherical
    system}. Then the list of spherically closed spherical roots has
  to be replaced with a longer list of spherical roots.

  The full list of spherical roots can be found in
  \cite[Table~30.2]{ti}.  The list in \cite[Table~2]{f4} only includes
  those spherical roots which can appear when $G$ is adjoint, but is
  still strictly longer than the list of spherically closed spherical
  roots. In that case, the small numbers are required to distinguish
  proportional spherical roots.
\end{remark}

\section{Gorenstein spherical Fano varieties}
\label{sec:gsfv}

A complete complex algebraic variety is called Gorenstein Fano if it
is normal and its anticanonical divisor is Cartier and ample. For
toric varieties one has a description of Gorenstein Fano varieties in
terms of convex geometry, namely one has a bijective correspondence
between Gorenstein toric Fano varieties and reflexive polytopes, \ie
lattice polytopes whose dual is a lattice polytope as well (see
\cite[Theorem 4.1.9]{Bat:DualPolyhedra}). Generalizing the notion of a
reflexive polytope, Pasquier established a similar correspondence for
horospherical varieties in terms of so-called $G/H$-reflexive
polytopes (see~\cite{Pasquier:FanoHorospherical}).  This
correspondence has recently been generalized to (arbitrary) Gorenstein
spherical Fano varieties (see~\cite{gsfv}).

In this section, we fix a spherical homogeneous space $G/H$.

\begin{definition}[{\cite[Definition~1.9]{gsfv}}]
  \label{def:qghrefl}
  A polytope $Q \subseteq \Nm_\Q$ is called \emph{$G/H$-reflexive} if
  the following conditions are satisfied:
  \begin{enumerate}
  \item $\rho( D ) / m_D \in Q$ for every $D \in \Dm$.
  \item $0 \in \topint(Q)$.
  \item Every vertex of $Q$ is contained in $\{ \rho( D ) / m_D : D
    \in \Dm \}$ or $\Nm \cap \Vm$.
  \item Consider the dual polytope
    \begin{align*}
      Q^* \coloneqq \bigcap_{u \in Q} \{ v \in \Mm_\Q : \langle u, v
      \rangle \ge -1 \}.
    \end{align*}
    A vertex $v \in Q^*$ such that $Q^* \cap ( v + \cone(\Sigma) ) =
    \{ v \}$ is called \emph{supported}
    (here, $\Sigma$ denotes the set of spherically closed spherical roots
    associated to $G/H$ according to Proposition~\ref{prop:clc}).
    Every supported vertex of
    $Q^*$ lies in the lattice $\Mm$.
  \end{enumerate}
  The set of supported vertices of $Q^*$ is denoted by
  $V_{\Supp}(Q^*)$.
\end{definition}

\begin{theorem}[{\cite[Theorem~1.10]{gsfv}}]
  \label{thm:bijection-Fano-poly}
  The assignment $X \mapsto \conv( \rho(D) / m_D : D \in \Delta)$
  induces a bijection between isomorphism classes of Gorenstein
  spherical Fano embeddings $G/H \hookrightarrow X$ and
  $G/H$-reflexive polytopes.
\end{theorem}

\begin{remark}
In the statement of Theorem~\ref{thm:bijection-Fano-poly}, recall that
$\Delta$ denotes the set of all $B$-invariant prime divisors
in $X$ and that we have a natural inclusion $\Dm \subseteq \Delta$ (by identifying
the colors with their closures). For $D \in \Delta \setminus \Dm$,
we have $m_D = 1$ and $\rho(D)/m_D = \rho(D) \in \Nm \cap \Vm$, so that
property (3) of Definition~\ref{def:qghrefl} is satisfied.
\end{remark}

\begin{example}
  \label{ex:ghrp}
  Let $G \coloneqq \SL(2) \times \C^*$. We fix a maximal torus
  contained in some Borel subgroup and denote by $\alpha$ the unique
  simple root of $\SL(2)$.  Moreover, we denote by $\varepsilon$ a
  primitive character of $\C^*$.  We define the spherically closed
  spherical system $\Ss \coloneqq (\Sigma, S^p, \Dm^a)$ with $\Sigma
  \coloneqq \{\alpha\}$, $S^p\coloneqq\emptyset$, $\Dm^a \coloneqq
  \{D_1, D_2\}$, and $\rho(D_1) \coloneqq \rho(D_2) \coloneqq
  \tfrac{1}{2} \alpha^\vee|_{\Lambda}$.  Let $\Mm$ be the lattice
  spanned by $b_1 \coloneqq \frac{1}{2}\alpha + \varepsilon$ and $b_2
  \coloneqq \frac{1}{2}\alpha - \varepsilon$. Then $(b_1, b_2)$ is a
  basis of $\Mm$, and we denote by $(b_1^*, b_2^*)$ the corresponding
  dual basis of the dual lattice $\Nm$.  We set $\rho'(D_1) \coloneqq
  b_1^*$ and $\rho'(D_2) \coloneqq b_2^*$.  Then $(\Mm, \rho')$ is an
  augmentation of $\Ss$, and we denote by $H \subseteq G$ the
  corresponding spherical subgroup.  The polytope $Q \coloneqq
  \conv(b_1^*, b_2^*, -b_1^*+b_2^*, -b_2^*) \subseteq \Nm_\Q$ is
  $G/H$-reflexive.  Its dual polytope is $Q^* = \conv(2b_1+b_2,
  -b_1+b_2, -b_1-b_2, -b_2) \subseteq \Mm_\Q$, and we have
  $V_{\Supp}(Q^*) = \{2b_1+b_2, -b_1+b_2\}$.  The polytopes are
  illustrated in Figure~\ref{fig:ghrp}.  The grey area is the
  valuation cone, the dashed arrows are $\rho'(D_1)$ and $\rho'(D_2)$,
  and the dotted arrows are translates of the spherically closed
  spherical root $\alpha = b_1 + b_2$ showing that exactly the circled
  vertices of $Q^*$ are supported. We denote by $G/H \hookrightarrow
  X$ the Gorenstein spherical Fano embedding associated to the
  $G/H$-reflexive polytope $Q$. The variety $X$ is isomorphic to
  $\Pd^1 \times \Pd^2$.
  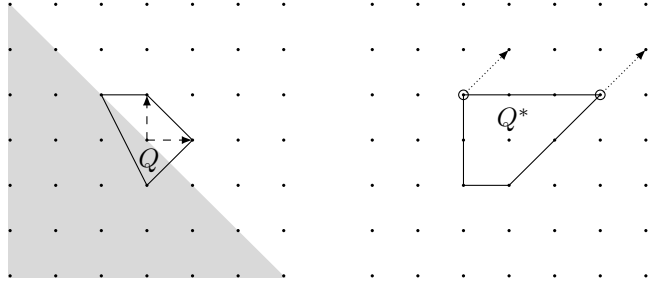
\begin{figure}[ht!]
    \begin{tikzpicture}[scale=0.6]
      \clip (-3.04, -3.04) -- (3.04, -3.04) -- (3.04, 3.04) -- (-3.04, 3.04) -- cycle;
      \fill[color=gray!30] (5, -5) -- (-5, -5) -- (-5, 5) -- cycle;
      \draw[draw=black] ( 0, -1 ) -- ( 1, 0 ) -- ( 0, 1 ) -- ( -1, 1 ) -- cycle;
      \foreach \x in {-4,...,4} \foreach \y in {-4,...,4} \fill (\x, \y) circle (1pt);
      \node at (0.05,-0.45) {$Q$};
      \draw[dashed,-latex] (0, 0) -- (1, 0);
      \draw[dashed,-latex] (0, 0) -- (0, 1);
    \end{tikzpicture}\hspace*{1cm}
    \begin{tikzpicture}[scale=0.6]
      \clip (-3.04, -3.04) -- (3.04, -3.04) -- (3.04, 3.04) -- (-3.04, 3.04) -- cycle;
      \draw[draw=black] ( 2, 1 ) -- ( -1, 1 ) -- ( -1, -1 ) -- ( 0, -1 ) -- cycle;
      \foreach \x in {-4,...,4} \foreach \y in {-4,...,4} \fill (\x, \y) circle (1pt);
      \draw (-1, 1) circle (3pt);
      \draw (2, 1) circle (3pt);
      \draw[densely dotted,-latex] (-1, 1) -- (0, 2);
      \draw[densely dotted,-latex] (2, 1) -- (3, 2);
      \node at (0.1,0.45) {$Q^*$};
    \end{tikzpicture}
    \caption{Illustration to Example~\ref{ex:ghrp}.}
    \label{fig:ghrp}
  \end{figure}
\end{example}

Our next aim is to determine the pseudo-index of a Gorenstein
spherical Fano variety $X$ with associated $G/H$-reflexive polytope
$Q$ combinatorially.  This can be done by translating
\cite[3.2]{sphmori} into the setting of $G/H$-reflexive polytopes. As
this works exactly as in the horospherical case (see~\cite[after
Lemma~5]{pas10}), we will be brief.  We write $u_D \coloneqq
\rho'(D)/m_D$ for every $D \in \adiv$.  For every $v \in
V_{\Supp}(Q^*)$ and $D \in \bdiv$ with $u_D \notin \widehat{v}$ 
(where $\widehat{v}$ is the
dual face to $v$, see the list of general notation) there
exists a rational curve $\cur_{D, v}$ with $(-K_X \cdot \cur_{D, v}) =
m_D + \langle \rho'(D), v \rangle$.  Moreover, for every edge $\ell$
of $Q^*$ connecting two supported vertices $v, w \in V_{\Supp}(Q^*)$
there exists a rational curve $\cur_{\ell}$.  Let $\chi_\ell \in \Mm$
be the uniquely determined primitive element such that $v-w$ is a
positive multiple of $\chi_\ell$.  Then we have $v-w = (-K_X \cdot
\cur_\ell) \cdot \chi_\ell$.  The families of curves $\cur_{D, v}$ and
$\cur_\ell$ generate the cone of effective one-cycles on $X$.

\begin{example}
  \label{ex:ix}
  We consider the Gorenstein spherical Fano variety $X$ of
  Example~\ref{ex:ghrp}.  Recall that we have $\rho'(D_1) = b_1^*$,
  $\rho'(D_2) = b_2^*$, and $V_{\Supp}(Q^*) = \{2b_1+b_2, -b_1+b_2\}$.
  We determine the pseudo-index $\iota_X$. We have
  \begin{align*}
    (-K_X \cdot \cur_{D_1,2b_1+b_2}) &= m_{D_1} + \langle \rho'(D_1), 2b_1+b_2\rangle = 1 + 2 = 3\text{,}\\
    (-K_X \cdot \cur_{D_2,2b_1+b_2}) &= m_{D_2} + \langle \rho'(D_2), 2b_1+b_2\rangle = 1 + 1 = 2\text{,}\\
    (-K_X \cdot \cur_{D_2,-b_1+b_2}) &= m_{D_2} + \langle \rho'(D_2),
    -b_1+b_2\rangle = 1 + 1 = 2\text{.}
  \end{align*}
  Note that $\cur_{D_1,-b_1+b_2}$ does not exist as $u_{D_1}$ lies in
  the dual face to $-b_1+b_2$.  Moreover, there exists exactly one
  edge $\ell$ of $Q^*$ connecting two supported vertices, and we
  have \begin{align*}(-K_X \cdot \cur_{\ell}) = 3\text{.}\end{align*}
  We conclude $\iota_X = \min\{2,3\} = 2$. As we have $\dim X = 3$ and
  $\rho_X = |\adiv| - \rank X= 2$, we see that
  Conjecture~\ref{conj:bcdd} holds for $X$.

  We can also determine $\is{X}$ and verify
  Conjecture~\ref{conj:esgp}.  We denote by $D_3$ and $D_4$ the
  $G$-invariant prime divisors in $X$ with $\rho'(D_3) = -b_1^*+b_2^*$
  and $\rho'(D_4) = -b_2^*$.  The intersection $Q^* \cap
  \cone(\Sigma)$ is the line segment connecting $0$ and $b_1+b_2$, and
  the supremum in Definition~\ref{def:is} is achieved for $\sv
  \coloneqq b_1+b_2$.  Hence we have
  \begin{align*}
    \is{X} = \langle \rho'(D_1), \sv \rangle + \langle \rho'(D_2), \sv
    \rangle + \langle \rho'(D_3), \sv \rangle + \langle \rho'(D_4),
    \sv \rangle = 1+1+0-1 = 1\text{.}
  \end{align*}
  Note that $\dim X - \rank X = 1$ and $X \cong \Pd^1 \times \Pd^2$ is
  isomorphic to a toric variety, therefore Conjecture~\ref{conj:esgp}
  holds for $X$.
\end{example}

We conclude this section with an observation, which will not be used
further.

\begin{prop}
  \label{prop:vert-col}
  Let $D \in \Dm$ with $u_D \notin \Vm$. Then $u_D$ is a vertex of
  $Q$.
\end{prop}
\begin{proof}
  The colors not contained in $\Vm$ are exactly the colors
  appearing in Table~\ref{tab:sr} (corresponding to non-dotted circles).
  According to Table~\ref{tab:sr}, there exists a spherically closed
  spherical root $\gamma \in \Sigma$ with $\langle u_D, \gamma \rangle
  > 0$.  Moreover, Table~\ref{tab:sr} also shows that there is at most
  one color $D' \in \Dm$ distinct from $D$ with $\langle u_{D'},
  \gamma \rangle > 0$ (otherwise set $D' \coloneqq D$) and that we
  then have $r \coloneqq \langle u_D, \gamma \rangle = \langle u_{D'},
  \gamma \rangle$.  Hence $H_\gamma \coloneqq \{u \in \Nm_\Q : \langle
  u, \gamma \rangle \le r\}$ is a supporting affine hyperplane for $Q$
  with $H_\gamma \cap Q = \conv(u_D, u_{D'})$. It follows that $u_D$
  and $u_{D'}$ are vertices of $Q$.
\end{proof}

\section{The generalized Mukai conjecture}
\label{sec:gener-mukai-conj}

The purpose of this section is to prove Theorem~\ref{th:bcddsph}.  Let
$X$ be a $\Q$-factorial Gorenstein spherical Fano variety, and let $Q$
be the associated $G/H$-reflexive polytope.  We define
\begin{align*}
  \varepsilon_X \coloneqq \min \rleft\{ m_D(1 + \langle u_D, v
  \rangle) : \text{$D \in \adiv$, $v \in V_{\Supp}( Q^* )$, $\langle
    u_D, v \rangle \ne -1$}\rright\}\text{.}
\end{align*}

\begin{remark}
  Our proof generalizes (and slightly simplifies) the approach of
  Casagrande (resp.~of Pasquier) for toric varieties (resp.~for
  horospherical varieties).  In the case of a horospherical variety
  $X$, our definition of $\varepsilon_X$ yields a value smaller than
  or equal to the corresponding definition of Pasquier.  Moreover, we
  are able to avoid an argument like \cite[Lemma~4]{cas06} or
  \cite[Lemma~5]{pas10}.
\end{remark}

\begin{prop}
  \label{p:ieq}
  We have $\iota_X \le \varepsilon_X$.
\end{prop}
\begin{proof}
  Let $v \in V_{\Supp}(Q^*)$ and $D \in \adiv$ with $\langle u_D, v
  \rangle \ne -1$, \ie $u_D \notin \widehat{v}$.  We first consider
  the case $D \in \Dm$. Then we have
  \begin{align*}
    (-K_X\cdot \cur_{D,v}) = m_D + \langle \rho'(D), v \rangle =
    m_D(1+\langle u_D, v\rangle)\text{.}
  \end{align*}

  Now consider the case $D \in \adiv \setminus \bdiv$.  Then we have
  $u_D \in \Vm$. As $\cone(\widehat{v})^\circ \cap \Vm \ne \emptyset$
  and $\cone(\widehat{v})$ is full-dimensional in $\Nm_\Q$ (as is the
  valuation cone $\Vm$), there exists $u' \in \cone(\widehat{v})^\circ
  \cap \Vm^\circ$. The line segment connecting $u'$ and $u_D$ lies in
  $\Vm^\circ$, except for possibly the point $u_D$. As $u_D \notin
  \widehat{v}$, the intersection of the line segment with the boundary
  of $\cone(\widehat{v})$ lies in a facet $\mu$ of
  $\cone(\widehat{v})$ which is not contained in the boundary of
  $\Vm$.  It follows that $\mu$ is a facet of exactly one other cone
  $\cone(\widehat{w})$ for some $w \in V_{\Supp}(Q^*)$ and the
  supported vertices $v$ and $w$ are joined by an edge $\ell$ of
  $Q^*$.  Let $\chi_\ell \in \Mm$ be the uniquely determined primitive
  element such that $v-w = (-K_X \cdot \cur_\ell) \cdot \chi_\ell$.
  Then we have $\langle u_D, \chi_\ell \rangle > 0$, hence $\langle
  \rho'(D), \chi_\ell\rangle$ is a positive integer, and
  \begin{align*}
    (-K_X \cdot \cur_{\ell}) 
      & = \big(\langle \rho'(D), -w\rangle + \langle \rho'(D), v\rangle\big)/\langle\rho'(D), \chi_\ell\rangle\\
      & \le \langle \rho'(D), -w\rangle + \langle \rho'(D), v\rangle\\
      & = m_D(\langle u_D, -w \rangle + \langle u_D, v \rangle) \\
      & \le m_D(1 + \langle u_D, v\rangle)\text{.}
  \end{align*}
\end{proof}

\begin{prop}
  \label{p:rfs}
  Let $\sv \in \conv(V_{\Supp}(Q^*))$.  Then we have
  \begin{align*}
    \varepsilon_X (|\adiv| - \rank X) \le \sum_{D \in \adiv} \big(m_D
    + \langle \rho'(D), \sv \rangle\big) \text{.}
  \end{align*}
\end{prop}
\begin{proof}
  There exist rational numbers $0 \le b_v \le 1$ for $v \in
  V_{\Supp}(Q^*)$ such that
  \begin{align*}
    \sum_{v \in V_{\Supp}(Q^*)} b_v = 1\text{,} && \sum_{v \in
      V_{\Supp}(Q^*)} b_v v = \sv\text{.}
  \end{align*}
  For every $D \in \adiv$ we define
  \begin{align*}
    A(D) \coloneqq \{v \in V_{\Supp}(Q^*) : \langle u_D, v\rangle =
    -1\}\text{.}
  \end{align*}
  Then we have
  \begin{align*}
    \langle u_D, \sv \rangle &= \sum_{v \in V_{\Supp}(Q^*)} b_v\langle u_D, v \rangle \\
    &= - \sum_{v \in A(D)} b_v + \sum_{v \notin A(D)} b_v \langle u_D,v\rangle \\
    &\ge -\sum_{v \in A(D)} b_v + \sum_{v \notin A(D)}
    b_v\rleft(\frac{\varepsilon_X}{m_D}-1\rright) \\
    &= \rleft(\frac{\varepsilon_X}{m_D}-1\rright) -
    \frac{\varepsilon_X}{m_D}\sum_{v \in A(D)} b_v \text{,}
  \end{align*}
  hence
  \begin{align*}
    \frac{\varepsilon_X - m_D}{\varepsilon_X} \le \sum_{v \in A(D)}
    b_v + \frac{m_D}{\varepsilon_X}\langle u_D, \sv \rangle\text{.}
  \end{align*}
  We sum the above inequality over all $D \in \adiv$ and obtain
  \begin{align*}
    |\adiv| - \frac{1}{\varepsilon_X}\sum_{D \in \adiv} m_D &\le
    \sum_{D \in \adiv}\sum_{v \in A(D)}
    b_v + \sum_{D \in \adiv} \frac{m_D}{\varepsilon_X} \langle u_D, \sv \rangle \\
    &= \rank X + \sum_{D \in \adiv} \frac{m_D}{\varepsilon_X} \langle
    u_D, \sv \rangle\text{,}
  \end{align*}
  where the last equality follows from the fact that $X$ is
  $\Q$-factorial.
  Indeed, the usual criterion for $\Q$-factoriality of a spherical
  variety (see, for instance, \cite[Theorem~3.2.14]{per14}) may be
  straightforwardly translated into the following condition:
  Every facet $\widehat{v}$ of $Q$
  for $v \in V_{\Supp}(Q^*)$ has exactly $\rank X$ vertices in $V(Q)$,
  where such a vertex can not be equal to $\rho'(D)$ for more than one
  $D \in \adiv$.
\end{proof}

\begin{cor}
  Let $\sv \in \conv(V_{\Supp}(Q^*))$.  Then we have
  \begin{align*}
    \rho_X(\iota_X-1) \le \sum_{D \in \adiv} \big(m_D - 1 + \langle
    \rho'(D), \sv \rangle\big) + \rank X\text{.}
  \end{align*}
  In particular, Proposition~\ref{prop:mineq} holds.
\end{cor}
\begin{proof}
  Recall that we have $\rho_X = |\adiv|-\rank X$.  We obtain
  \begin{align*}
    \rho_X(\iota_X-1) &\le \rho_X(\varepsilon_X-1)\\
    &= ( |\adiv| - \rank X ) \varepsilon_X
    - |\adiv|+\rank X\\
    &\le \sum_{D \in \adiv} \big(m_D
    + \langle \rho'(D), \sv \rangle\big) - |\adiv|+\rank X\\
    & = \sum_{D \in \adiv} \big(m_D - 1 + \langle \rho'(D), \sv \rangle\big) + \rank X\text{,}
  \end{align*}
  where the two inequalities follow from Proposition~\ref{p:ieq} and
  Proposition~\ref{p:rfs}.
\end{proof}

\begin{proof}[Proof of Theorem~\ref{th:bcddsph}]
  According to \cite[Lemma~13.3]{gsfv}, there exists an element $\sv
  \in V_{\Supp}(Q^*) \cap \cone(\Sigma)$.  Then
  Definition~\ref{def:is} and Conjecture~\ref{conj:esgp} imply
  \begin{align*}
    \sum_{D \in \adiv}\big(m_D - 1 + \langle \rho'(D), \sv\rangle\big)
    \le \is{X} \le \dim X - \rank X\text{,}
  \end{align*}
  so that Proposition~\ref{prop:mineq} implies $\rho_X(\iota_X-1) \le
  \dim X$.

  Now assume $\rho_X(\iota_X-1) = \dim X$. Then we have $\is{X} = \dim
  X - \rank X$, and according to Conjecture~\ref{conj:esgp} the
  variety $X$ is isomorphic to a toric variety. Hence we have $X \cong
  (\Pd^{\iota_X-1})^{\rho_X} $ by \cite[Theorem~1(ii)]{cas06}.
\end{proof}

\section{Spherical skeletons}
\label{sec:spherical-skeletons}

The aim of this section is to restate Conjecture~\ref{conj:esgp} in
terms of the information captured in the following notion, \ie in a
purely combinatorial way.

\begin{definition}
  A \emph{spherical skeleton} is a quadruple $\Rs \coloneqq (\Sigma,
  S^p, \Dm^a, \gdiv)$ where $(\Sigma, S^p, \Dm^a)$ is a spherically
  closed spherical system and $\gdiv$ is an abstract
  finite set equipped with a
  map $\rho\colon \gdiv \to \Lambda^*$ such that for every $D \in \gdiv$,
  $\gamma \in \Sigma$ we have $\langle \rho(D), \gamma \rangle \le 0$.
  A spherical skeleton determines a finite set $\adiv \coloneqq \bdiv
  \cup \gdiv$ equipped with a map $\rho\colon \adiv \to \Lambda^*$.  A
  spherical skeleton is called \emph{complete} if $\cone(\rho(D) : D
  \in \adiv) = \Lambda^*_\Q$.
\end{definition}

To any spherical variety $X$ we associate a spherical skeleton $\Rs_X
\coloneqq (\Sigma, S^p, \Dm^a, \gdiv)$ in the obvious way, \ie
$(\Sigma, S^p, \Dm^a)$ is the spherically closed system associated to
the open $G$-orbit in $X$ and $\gdiv \coloneqq \adiv \setminus \bdiv$
is the set of $G$-invariant prime divisors in $X$.  If $X$ is
complete, then $\Rs_X$ is also complete.  

\begin{remark}
  \label{rem:Cox-ring-vs-sph-skeleton}
  Let $X$ be a complete spherical $G$-variety. According to
  \cite[4.3.2]{brcox}, the Cox ring $\Rm(X)$, considered without its
  grading, is determined by the spherical skeleton $\Rs_X$ (see
  also~\cite[Remark~5.4]{gag14}). We are going to take crucial
  advantage of this fact in the next section.
\end{remark}

Let $\Rs_1 \coloneqq (\Sigma_1, S^p_1, \Dm_1^a, \gdiv_1)$ and $\Rs_2
\coloneqq (\Sigma_2, S^p_2, \Dm_2^a, \gdiv_2)$ be spherical skeletons,
where the underlying root system of the spherically closed spherical
system of $\Rs_1$ (resp.~of $\Rs_2$) is $R_1$ (resp.~$R_2$).  The
spherical skeletons $\Rs_1$ and $\Rs_2$ will be considered to be
\emph{isomorphic}, written $\Rs_1 \cong \Rs_2$, if there exists an
isomorphism of root systems $\varphi_R\colon R_1 \to R_2$ as well as
bijections $\varphi_\bdiv\colon \Dm^a_1 \to \Dm^a_2$ and
$\varphi_\gdiv\colon \gdiv_1 \to \gdiv_2$ such that
$\varphi_R(\Sigma_1) = \Sigma_2$, $\varphi_R(S^p_1) = S^p_2$, and
$\rho_1(D) = \rho_2(\varphi_\bdiv(D)) \circ \varphi_R$ on $\Sigma_1$ for
every $D \in \Dm^a_1$
as well as $\rho_1(D) = \rho_2(\varphi_\gdiv(D)) \circ \varphi_R$ on $\Sigma_1$
for every $D \in \gdiv_1$. If
$R_1 = R_2$ and $\varphi_R = \operatorname{id}$, then we will write
$\Rs_1 = \Rs_2$ for $\Rs_1 \cong \Rs_2$.

\begin{definition}
  For a spherical skeleton $\Rs \coloneqq (\Sigma, S^p, \Dm^a, \gdiv)$
  we set
  \begin{align*}
    \Qm^*_\Rs \coloneqq \bigcap_{D\in\adiv} \{v \in \Lambda_\Q :
    \langle \rho(D), v\rangle \ge -m_D \}
  \end{align*}
  and define
  \begin{align*}
    \is{\Rs} \coloneqq \sup{} \rleft\{ \sum_{D \in \adiv} \rleft( m_D
    - 1 + \langle \rho(D), \sv \rangle \rright): \sv \in \Qm^*_\Rs
    \cap \cone(\Sigma)\rright\} \in \Q_{\ge0} \cup \{\infty\}\text{.}
  \end{align*}
\end{definition}

Note that we have
$\Qm^*_{\Rs_X} = Q^*_X \cap \lspan_\Q \Sigma$, hence $\is{X} =
\is{\Rs_X}$.

\begin{example}
  We determine the spherical skeleton $\Rs_X = (\Sigma, S^p, \Dm^a,
  \gdiv)$ in the situation of Example~\ref{ex:ix}. We have $\Sigma =
  \{\alpha\}$, $S^p = \emptyset$, $\Dm^a = \{D_1, D_2\}$, and $\gdiv =
  \{D_3, D_4\}$ with $\rho(D_1) = \rho(D_2) =
  \frac{1}{2}\alpha^\vee|_{\Lambda}$, $\rho(D_3) = 0 \in \Lambda^*$,
  and $\rho(D_4) = -\frac{1}{2}\alpha^\vee|_{\Lambda}$.
\end{example}

A \emph{multiplicity-free space} is a complex vector space $V$
equipped with a linear $G$-action such that $V$ is also a spherical
variety.  Multiplicity-free spaces have been classified
(see~\cite{sr-benrat, sr-leahy}). The spherical skeletons $\Rs_V$ for
all multiplicity-free spaces $V$ can be obtained from
\cite[Section~2]{gcscsv}. Note that in these cases $\Rs_V$ is complete
while $V$ is not.  We can now restate Conjecture~\ref{conj:esgp} in a
combinatorial way as follows.

\begin{conj}
  \label{conj:esgp2}
  Let $\Rs \coloneqq (\Sigma, S^p, \Dm^a, \Gamma)$ be a complete
  spherical skeleton.  Then we have
  \begin{align*}
    \is{\Rs} \le |R^+ \setminus R_{S^p}^+|\text{,}
  \end{align*}
  where equality holds if and only if $\Rs$ is isomorphic to the
  spherical skeleton $\Rs_V$ of a multiplicity-free space $V$.
\end{conj}

The equivalence of Conjecture~\ref{conj:esgp2} and
Conjecture~\ref{conj:esgp} will be shown in the next section. We
conclude the present section by explaining why $\is{\Rs}$ is always
nonnegative and rational.

\begin{prop}
  Let $\Rs \coloneqq (\Sigma, S^p, \Dm^a, \Gamma)$ be a spherical
  skeleton.  Then we have $\is{\Rs} \in \Q_{\ge 0} \cup \{\infty\}$.
\end{prop}
\begin{proof}
  We identify $\Lambda_\Q \otimes_\Q \R \cong \R^\Sigma$.  We define
  $c \in \R^\Sigma$ by $c_\gamma \coloneqq \sum_{D \in \Delta}
  \langle\rho(D), \gamma\rangle$ as well as $b \in \R^\adiv$ by $b_D
  \coloneqq m_D$ and denote by $A \in \R^{\adiv \times \Sigma}$ the
  matrix with $A_{D,\gamma} \coloneqq \langle-\rho(D), \gamma\rangle$.
  Then we have
  \begin{align*}
    \is{\Rs} = \sum_{D\in\adiv} (m_D-1) + \sup{} \rleft\{c^tx :
    \text{$x\in\R^\Sigma$, $Ax \le b$, $x \ge 0$}\rright\}\text{,}
  \end{align*}
  where the supremum is the optimal value of a linear program.  Assume
  $\is{\Rs} < \infty$.  Passing to the dual linear program, we obtain
  \begin{align*}
    \is{\Rs} = \sum_{D\in\adiv} (m_D-1) + \inf{}\rleft\{b^ty :
    \text{$y \in \R^{\adiv}$, $A^ty \ge c$, $y \ge
      0$}\rright\}\text{.}
  \end{align*}
  As $b^ty = \sum_{D\in\adiv}b_Dy_D$ and $b_D = m_D > 0$, we get
  $\is{\Rs} \ge 0$. As we can find a solution $x$ of the original linear
  program which is a vertex of $\Qm^*_{\Rs} \cap \cone(\Sigma)$, we obtain
  $\is{\Rs} \in \Q$.
\end{proof}

\section{Cox rings and multiplicity-free spaces}
\label{sec:cox-rings-mult}

The purpose of this section is to prove the equivalence of
Conjecture~\ref{conj:esgp} and Conjecture~\ref{conj:esgp2}.  We will
make fundamental use of the fact that spherical varieties have
finitely generated Cox rings, a result due to Brion (see~\cite{brcox},
see also \cite[Theorem 4.3.1.5]{coxrings}).  For details and
references on Cox rings, we refer the reader to \cite{coxrings}.

If $X$ is a complete spherical $G$-variety, the following notation
will be used throughout this section. We fix a Borel subgroup $B
\subseteq G$ and a maximal torus $T \subseteq B$ and write $S
\subseteq R$ for the induced set of simple roots of the root system $R
\subseteq \Xf(T)$. We denote by $\Rm(X)$ the Cox ring of $X$.  There
exists an open subset $\widehat{X} \subseteq \overline{X} \coloneqq
\Spec \Rm(X)$ with complement of codimension at least $2$ such that
there exists a good quotient $\pi\colon \widehat{X} \to X$ by the
quasitorus $\Td \coloneqq \Spec \C[\Cl(X)]$
(see~\cite[Construction~1.6.3.1]{coxrings}).

A quasitorus (also called a diagonalizable group) is
an algebraic group isomorphic to a closed subgroup of a torus
and is determined by its character group. The
character group of $\Td$ is $\Cl(X)$, the divisor class group of $X$.

According to
\cite[Proposition~4.2.3.2]{coxrings}, there is a finite epimorphism
$G' \to G$ of connected algebraic groups and a $G'$-action on
$\widehat{X}$ (hence also on $\overline{X}$) commuting with the
$\Td$-action such that $\pi \colon \widehat{X} \to X$ becomes
$\overline{G}$-equivariant where $\overline{G} \coloneqq G' \times
\Td^\circ$ acts on $X$ via the epimorphism $\varepsilon \colon
\overline{G} \to G$. There is a Borel subgroup $\overline{B}$ of
$\overline{G}$ and a maximal torus $\overline{T} \subseteq
\overline{B}$ such that $\overline{B}$ is mapped onto $B$ and
$\overline{T}$ is mapped onto $T$ by $\varepsilon$.

As $\pi \colon \widehat{X} \to X$ is a $\overline{G}$-equivariant
geometric quotient (see \cite[Corollary 1.6.2.7]{coxrings}) over the
set of smooth points, the preimage of the open $\overline{B}$-orbit in
$X$ is a finite union of $\overline{B}$-orbits in $\widehat{X}$, hence
an open $\overline{B}$-orbit.  Therefore the $\overline{G}$-variety
$\widehat{X}$ (hence also $\overline{X}$) is spherical.  Moreover, we
have $\dim \overline{X} = \dim X + \dim \Td$ and the stabilizers of
the open $\overline{B}$-orbits in $\overline{X}$ and $X$ coincide.

We identify the root system (and the set of simple roots) of $G$ and
$\overline{G}$ via $\varepsilon$.  Then the spherical skeleton $\Rs_X$
does not depend on whether we consider $X$ as a $G$-variety or a
$\overline{G}$-variety.  Moreover, $\Rs_{\overline{X}}$ is also a
spherical skeleton for the root system $R$, so that we may ask whether
$\Rs_{\overline{X}} = \Rs_X$ holds.

\begin{prop}
  \label{prop:lincox}
  Let $X$ be a complete spherical $G$-variety.  The variety $X$ is
  isomorphic to a toric variety if and only if $\overline{X}$ is
  $\overline{G}$-equivariantly isomorphic to a multiplicity-free
  space.
\end{prop}
\begin{proof}
  \enquote{$\Rightarrow$}: As $X$ is isomorphic to a toric variety,
  the Cox ring $\Rm(X)$ is isomorphic to a polynomial ring
  (see~\cite{tcox}), hence $\overline{X} \cong \Ad^n$.  As
  $\overline{X}$ is an affine spherical variety, it contains exactly
  one closed $\overline{G}$-orbit.  It follows from the \'etale slice
  theorem of Luna (see~\cite[Proposition~5.1]{kp85}) that the action
  of $\overline{G}$ on $\overline{X} \cong \Ad^n$ is linearizable, \ie
  there is a multiplicity-free $\overline{G}$-space $\Ad^n$ and a
  $\overline{G}$-equivariant isomorphism $\overline{X} \cong \Ad^n$.

  \enquote{$\Leftarrow$}: As $\overline{X} \cong \Ad^n$, the Cox ring
  $\Rm(X)$ is a polynomial ring. As $X$ is complete, it follows from
  \cite[Exercise~3.9(1)]{coxrings} that $X$ is isomorphic to a toric
  variety.
\end{proof}

For any combinatorial object associated to $X$ we denote the one
associated to $\overline{X}$ by the same symbol with an added bar, \eg
we have $\overline{\rho} \colon \overline{\Delta} \to
\overline{\Lambda}{}^*$.

Since $\widehat{X} \subseteq \overline{X}$ has complement of
codimension at least $2$, we can naturally identify the
$\overline{B}$-invariant Weil divisors in $\overline{X}$ and
$\widehat{X}$.  We say that $\pi$ induces a bijection $\pi^*\colon
\adiv \to \overline{\adiv}$ if the support of the pullback under $\pi$
of every $D \in \adiv$ is irreducible and this defines a bijection
$\adiv \to \overline{\adiv}$.

\begin{remark}
  \label{rem:bs}
  Assume that $\pi$ induces a bijection $\pi^*\colon \adiv \to
  \overline{\adiv}$. Then this bijection restricts to a bijection
  $\pi^*\colon \bdiv \to \overline{\bdiv}$.  It follows from
  \cite[Lemme~5.2]{camus}, see also \cite[Lemma~3.1]{gcscsv}, that
  $\Ss_{\overline{X}} = \Ss_X$.
\end{remark}

\begin{remark}
  If $\Cl(X)$ is free, then $\Td$ is connected (\ie a torus) and the
  map $\pi \colon \widehat{X} \to X$ is a locally trivial
  $\Td$-principal bundle (with respect to the Zariski topology) over
  $X_{\reg}$ (see, for instance, \cite[Exercise~1.17(1)]{coxrings}).
  In this situation, it is straightforward to show that $\pi$ induces
  a bijection $\pi^*\colon \adiv \to \overline{\adiv}$ (hence
  $\Ss_{\overline{X}} = \Ss_X$ by Remark~\ref{rem:bs}) and that
  $\overline{\rho}(\pi^*(D)) = \rho(D)$ for every $D \in \gdiv$.
  Therefore we have $\Rs_{\overline{X}} = \Rs_X$.
\end{remark}

\begin{example}
  \label{cex:sph-sys-stays-same}
  If $\Cl(X)$ is not free, then we may have $\Rs_{\overline{X}} \ne
  \Rs_X$: Let $G \coloneqq \SL_2 \times \C^*$.  We fix a maximal torus
  contained in some Borel subgroup and denote by $\alpha$ the unique
  simple root of $\SL(2)$.  Moreover, we denote by $\varepsilon$ a
  primitive character of $\C^*$.  We define the spherically closed
  spherical system $\Ss \coloneqq (\Sigma, S^p, \Dm^a)$ with $\Sigma
  \coloneqq \{2\alpha\}$, $S^p\coloneqq\emptyset$, and $\Dm^a
  \coloneqq \emptyset$.  Let $\Mm$ be the lattice spanned by $b_1
  \coloneqq \alpha + \varepsilon$ and $b_2 \coloneqq \alpha -
  \varepsilon$. Then $(b_1, b_2)$ is a basis of $\Mm$, and we denote
  by $(b_1^*, b_2^*)$ the corresponding dual basis of the dual lattice
  $\Nm$.  Then $(\Mm, \emptyset)$ is an augmentation of $\Ss$, and we
  denote by $H \subseteq G$ the corresponding spherical subgroup.
 
  The polytope $Q \coloneqq \conv(b_1^*+b_2^*, -b_1^*+b_2^*,
  -b_1^*-b_2^*, b_1^*-b_2^*) \subseteq \Nm_\Q$ is $G/H$-reflexive.
  Its dual polytope is $Q^* = \conv(b_1, b_2, -b_1, -b_2) \subseteq
  \Mm_\Q$, The polytopes are illustrated in Figure~\ref{fig:sl2n} (as
  in Example~\ref{ex:ghrp}).  We denote by $G/H \hookrightarrow X$ the
  Gorenstein spherical Fano embedding associated to the
  $G/H$-reflexive polytope $Q$.
  \begin{figure}[ht!]
    \begin{tikzpicture}[scale=0.6]
      \clip (-2.04, -2.04) -- (2.04, -2.04) -- (2.04, 2.04) -- (-2.04, 2.04) -- cycle;
      \fill[color=gray!30] (-5, 5) -- (-5, -5) -- (5, -5) -- cycle;
      \draw[draw=black] ( 1, 1 ) -- ( -1, 1 ) -- ( -1, -1 ) -- ( 1, -1 ) -- cycle;
      \foreach \x in {-4,...,4} \foreach \y in {-4,...,4} \fill (\x, \y) circle (1pt);
      \node at (-0.4,0) {$Q$};
      \draw[dashed,-latex] (0, 0) -- (1, 1); 
    \end{tikzpicture}\hspace*{1cm}
    \begin{tikzpicture}[scale=0.6]
      \clip (-2.04, -2.04) -- (2.04, -2.04) -- (2.04, 2.04) -- (-2.04, 2.04) -- cycle;
      \draw[draw=black] ( -1, 0 ) -- ( 0, -1 ) -- ( 1, 0 ) -- ( 0, 1 ) -- cycle;
      \foreach \x in {-4,...,4} \foreach \y in {-4,...,4} \fill (\x, \y) circle (1pt);
      \draw (1, 0) circle (3pt);
      \draw (0, 1) circle (3pt);
      \draw[densely dotted,-latex] (1, 0) -- (2, 1);
      \draw[densely dotted,-latex] (0, 1) -- (1, 2);
      \node at (0.5,0) {$Q^*$};
    \end{tikzpicture}
    \caption{Illustration to Example~\ref{cex:sph-sys-stays-same}.}
    \label{fig:sl2n}
  \end{figure}
  Then $\Cl(X) \cong \Z^2 \oplus \Z/2\Z$ has torsion, \ie $\Td$ is not
  connected. According to \cite[4.3.2]{brcox} (the computation is
  similar to \cite[Example~5.5]{gag14}), we have
  \begin{align*}
    \Rm(X) \cong \C[ X_1, X_2, X_3, W_1, W_2, W_3 ] / \langle X_{1}
    X_{2} - X_{3}^2 - W_2^2 \rangle\text{.}
  \end{align*}
  It is straightforward to check that we may take $\overline{G}
  \coloneqq \SL(2) \times (\C^*)^3$ acting on $\overline{X}$ with the
  divisor of $X_1$ consisting of two $\overline{B}$-invariant, but not
  $\overline{G}$-invariant, irreducible components. Hence we have
  $\big|\overline{\bdiv}\big| \ge 2$.  In particular,
  $\Rs_{\overline{X}} \ncong \Rs_X$.
\end{example}

\begin{lemma}
  \label{lem:primitive-rhos}
  Assume that $\pi$ induces a bijection $\pi^* \colon \adiv \to
  \overline{\adiv}$ and that $\overline{\rho}(D) \in
  \overline{\Lambda}{}^*$ is either $0$ or primitive for every $D \in
  \overline{\Gamma}$.  Then we have $\Rs_{\overline{X}} = \Rs_X$.
\end{lemma}
\begin{proof}
  According to Remark~\ref{rem:bs}, we have $\Ss_{\overline{X}} =
  \Ss_X$.  In particular, we identify $\overline{\Lambda} = \Lambda$.
  The bijection $\pi^*\colon \adiv \to \overline{\adiv}$ restricts to
  a bijection $\pi^*\colon \gdiv \to \overline{\gdiv}$.  Let
  $\pi_*\colon \overline{\Nm}_\Q \to \Nm_\Q$ be the map dual to the
  inclusion $\pi^*\colon \Mm_\Q \to \overline{\Mm}_\Q$ induced by
  $\pi$.  It follows from \cite[Theorem 4.1]{knopsph} that for every
  $D \in \gdiv$ we have $\pi_*(\overline{\rho}'(\pi^*(D))) =
  k_D\rho'(D)$ for some positive integer $k_D$.  As the restriction of
  $\pi^*\colon \Mm_\Q \to \overline{\Mm}_\Q$ to $\Lambda$ is the
  identity, we obtain $\overline{\rho}(\pi^*(D)) = k_D \rho(D)$, for
  every $D \in \gdiv$. Hence we have either $\overline{\rho}(\pi^*(D))
  = \rho(D) = 0$ or $k_D = 1$ as $\overline{\rho}(\pi^*(D))$ is
  primitive.
\end{proof}

\begin{prop}
  \label{prop:qt}
  Let $X$ be a complete spherical $G$-variety.  Assume that the
  variety $X$ is isomorphic to a toric variety or, equivalently, that
  $\overline{X}$ is a multiplicity-free space.  Then we have
  $\Rs_{\overline{X}} = \Rs_X$.
\end{prop}
\begin{proof}
  As $\overline{X}$ is an affine spherical variety, it contains
  exactly one closed $\overline{G}$-orbit $\overline{Y}$, which is a
  point (as every multiplicity-free space has the origin as fixed
  point).  It follows from \cite[Theorem 6.6]{knopsph} that the set
  $\{\overline{\rho}'(D) : D \in \overline{\adiv}\}$ spans
  $\overline{\Nm}_\Q$ as vector space. As $\overline{X}$ has a
  factorial coordinate ring, the above set is also linearly
  independent. We obtain $\big|\overline{\Delta}\big| = \rank
  \overline{X}$.

  As the stabilizers of the open $\overline{B}$-orbits in
  $\overline{X}$ and $X$ coincide, we obtain $\rank \overline{X} =
  \rank X + \dim \Td$ from \cite[Theorem 6.6]{knopsph}.  By
  \cite[Proposition 4.1.1]{brcox}, we have
  \begin{align*}
    \dim \Td = \rank \Cl(X) = |\adiv| - \rank X\text{,}
  \end{align*}
  which implies $\big|\overline{\adiv}\big| = |\adiv|$.

  As $\pi$ is surjective, for every $D \in \adiv$ the preimage
  $\pi^{-1}(D)$ contains at least one irreducible component mapping
  dominantly to the irreducible divisor $D$. It follows that $\pi$ induces a
  bijection $\pi^*\colon \adiv \to \overline{\adiv}$.

  Every multiplicity-free space can be written as product
  of trivial and indecomposable multiplicity-free spaces
  in the sense of \cite[Definition~1.1]{gcscsv}.
  It follows that for every $D \in \overline{\gdiv}$
  we have either $\overline{\rho}(D) = 0$ or $D$
  corresponds to a symbol \enquote{$\gamma$} occurring in 
  the list in \cite[Section~2]{gcscsv}, which means
  $\langle \overline{\rho}(D), \gamma \rangle = -1$
  for the corresponding spherically closed spherical root $\gamma 
  \in \overline{\Lambda}$,
  in particular $\overline{\rho}(D)$ is primitive in 
  $\overline{\Lambda}{}^*$.

  The result now follows from Lemma \ref{lem:primitive-rhos}.
\end{proof}

\begin{theorem}
  \label{th:vp}
  Let $X$ be a complete spherical $G$-variety. Then the following
  statements are equivalent.
  \begin{enumerate}
  \item The variety $X$ is isomorphic to a toric variety.
  \item There exists a multiplicity-free space $V$ such that $\Rs_X
    \cong \Rs_V$.
  \end{enumerate}
\end{theorem}
\begin{proof}
  \enquote{$\Rightarrow$}: This direction follows from
  Proposition~\ref{prop:lincox} and Proposition~\ref{prop:qt}

  \enquote{$\Leftarrow$}: As there exists a multiplicity-free
  $G'$-space $V \cong \C^{n+1}$ with $\Rs_X \cong \Rs_V$, we obtain a
  projective space $\Pd^{n} \coloneqq V/\C^*$ with spherical
  $G'$-action such that $\Rs_X \cong \Rs_{\Pd^{n}}$. According to
  Remark \ref{rem:Cox-ring-vs-sph-skeleton}, $\Rm(X) \cong \Rm(\Pd^n)$
  is a polynomial ring, and $X$ is isomorphic to a toric variety.
\end{proof}

\begin{lemma}
  \label{le:c2p}
  Let $Y_2$ be a spherical $G$-variety with $\Gamma(Y_2, \Om_{Y_2}) =
  \C$ such that every $G$-orbit in $Y_2$ is of codimension at most
  $1$. Then there exists a finite epimorphism $G' \to G$ of connected
  algebraic groups and a $G'$-equivariant open embedding $Y_2
  \hookrightarrow Y$ into a projective spherical $G'$-variety $Y$ such
  that $Y \setminus Y_2$ is of codimension at least $2$ in $Y$.
\end{lemma}
\begin{proof}
  According to \cite[Remark~3.3.4.2]{coxrings} (which requires the
  assumption $\Gamma(Y_2, \Om_{Y_2}) = \C$), there exists at least one
  projective (hence complete) normal variety $Y$ with the same
  (graded) Cox ring as $Y_2$, hence the varieties $Y$ and $Y_2$
  coincide up to a subset of codimension at least $2$.  As the Cox
  rings of $Y$ and $Y_2$ coincide, we have $\overline{Y} =
  \overline{Y_2}$. The following commutative diagram illustrates the
  situation.
  \begin{align*}
    \xymatrix{
      \overline{Y_2}\ar@{=}[rr] & & \overline{Y}\\
      \widehat{Y_2} \ar@{^{(}->}[u] \ar[d] & \ar@{_{(}->}[l] 
      \widehat{Z} \coloneqq \widehat{Y_2} \cap \widehat{Y} \ar[d]
      \ar@{^{(}->}[r] & \widehat{Y} \ar@{_{(}->}[u] \ar[d] \\
      Y_2 & \ar@{_{(}->}[l] Z \ar@{^{(}->}[r] & Y
    }
  \end{align*}
  As explained at the beginning of this section, $\overline{Y_2}$
  becomes a spherical $\overline{G}$-variety.  According to
  \cite[Corollary~3.1.4.6]{coxrings} the open subset $\widehat{Y}
  \subseteq \overline{Y} = \overline{Y_2}$ is
  $\overline{G}$-invariant, which naturally makes $Y$ a spherical
  $G'$-variety. We may naturally consider $Y_2$ as a spherical
  $G'$-variety.  The intersection $\widehat{Z}$ as defined in the
  diagram is $\overline{G}$-invariant and an open subset with
  complement of codimension at least $2$ in $\overline{Y}$.  We obtain
  a $G'$-invariant open subset $Z$ of $Y$ and $Y_2$ with complement of
  codimension at least $2$ in $Y$ and $Y_2$ respectively.  As every $G'$-orbit in
  $Y_2$ has codimension at most $1$, we obtain $Y_2 = Z$ and then $Y_2
  \hookrightarrow Y$ as required.
\end{proof}

\begin{definition}
  Let $\Rs \coloneqq (\Sigma, S^p, \Dm^a, \Gamma)$ be a spherical
  skeleton. We define the spherical skeleton $[\Rs] \coloneqq (\Sigma,
  S^p, \Dm^a, [\Gamma])$ where $[\Gamma] \coloneqq \{D \in \Gamma :
  \rho(D) \ne 0 \in \Lambda^*\}$.  For two spherical skeletons
  $\Rs_1$, $\Rs_2$ we write $\Rs_1 \sim \Rs_2$ if $[\Rs_1] \cong
  [\Rs_2]$.
\end{definition}

\begin{remark}
  \label{rem:sim}
  Let $\Rs_1 \sim \Rs_2$ be two spherical skeletons.  It is not
  difficult to see that $\is{\Rs_1} = \is{\Rs_2}$. Moreover,
  $\Rs_1$ is the spherical skeleton of a multiplicity-free
  space if and only if $\Rs_2$ is so: If $D \in \Gamma$
  corresponds to a symbol \enquote{$\gamma$} occurring in the list in
  \cite[Section~2]{gcscsv}, then we have $\rho(D) \ne 0$
  (see also the proof of Proposition~\ref{prop:qt}). Therefore
  any $D \in \Gamma$ with $\rho(D) = 0$ can only come from a trivial factor ($\C^*$ acting
  on $\C$ by scalar multiplication).
\end{remark}

\begin{theorem}
  \label{th:conjeq}
  Conjecture~\ref{conj:esgp} is equivalent to
  Conjecture~\ref{conj:esgp2}
\end{theorem}
\begin{proof}
  By Theorem \ref{th:vp} and the fact that $\dim X - \rank X = \dim
  G/P = |R^+ \setminus R_{S^p}^+|$ where $P$ is the stabilizer of the
  open $B$-orbit in $X$ (see~\cite[Theorem~6.6]{knopsph}), it is clear
  that Conjecture~\ref{conj:esgp2} implies Conjecture~\ref{conj:esgp}.
  It remains to show the other implication.

  Assume Conjecture~\ref{conj:esgp}, let $R$ be a root system, and let
  $\Rs \coloneqq (\Sigma, S^p, \Dm^a, \Gamma)$ be a complete spherical
  skeleton. Our goal is to construct a complete spherical $G'$-variety
  $Y$ for some connected reductive group $G'$ with root system $R$
  such that $\Rs_Y \sim \Rs$.

  We denote by $G^{ss}$ the semisimple simply-connected algebraic
  group with root system $R$. We first construct an augmentation of
  the spherical system $\Ss \coloneqq (\Sigma, S^p, \Dm^a)$ for the
  reductive group $G \coloneqq G^{ss} \times (\C^*)^\gdiv$, where both
  the character and cocharacter lattices of $(\C^*)^\gdiv$ are
  naturally identified with $\Z^\gdiv$.  Let $\Mm \coloneqq \Lambda
  \oplus \Z^\gdiv \subseteq \Xf(B)$, $\Nm \coloneqq \Hom(\Mm, \Z)
  \cong \Hom(\Lambda, \Z) \oplus \Z^\gdiv$, and $\rho'\colon \bdiv \to
  \Nm$, $D \mapsto (\rho(D), 0)$.  Then $(\Mm, \rho')$ is an
  augmentation of $\Ss$ for $G$, and we denote the corresponding
  spherical homogeneous space by $O$.

  The elements $\rho'(D) \coloneqq (\rho(D), e_D) \in \Nm$ for $D \in
  \Gamma$, where $(e_D)_{D\in\gdiv}$ denotes the standard basis of
  $\Z^\gdiv$, are primitive and distinct, hence correspond to certain
  $G$-orbits of codimension 1, which can be added to $O$ resulting in
  a spherical embedding $O \hookrightarrow Y_1$.  It is clear that
  $\Rs_{Y_1} = \Rs$. As $\Rs$ is complete, we can choose finitely many
  primitive elements in $\Sigma^\perp \cap \Nm$ which, together with
  the $\rho'(D)$ for $D \in \adiv$, span the whole space $\Nm_\Q$ as
  convex cone. The corresponding $G$-orbits of codimension 1 can be
  further added to $Y_1$, resulting in a spherical variety $Y_2$ such
  that $Y_1 \hookrightarrow Y_2$.  Then we have $\Rs_{Y_2} \sim
  \Rs_{Y_1} = \Rs$, and $Y_2$ still consists of $G$-orbits of
  codimension at most 1.  Moreover, we have $\Gamma(Y_2, \Om_{Y_2}) =
  \C$, which follows from the fact that the $\rho'(D)$ for $D \in
  \adiv$ span the whole space $\Nm_\Q$ as convex cone.  Hence we may
  use Lemma~\ref{le:c2p} to obtain a completion $Y_2 \hookrightarrow
  Y$ with $\Rs_Y = \Rs_{Y_2} \sim \Rs$

  According to Remark~\ref{rem:sim} and Theorem~\ref{th:vp}, we have
  \begin{align*}
    \is{\Rs} = \is{\Rs_Y} = \is{Y} \le \dim Y - \rank Y = |R^+
    \setminus R_{S^p}^+|
  \end{align*}
  where equality holds if and only if $\Rs$ is isomorphic to the
  spherical skeleton of a multiplicity-free space.
\end{proof}

\section{A smoothness criterion}

Before giving the proof of Theorem~\ref{th:smoothness}, we explain how
it can be used to determine the smoothness of an arbitrary spherical
variety $X$ when the combinatorial data is known. This is a straightforward
translation of \cite[Proposition~6.1]{camus} (see also \cite[Section~3]{gcscsv})
into the setting of spherical skeletons.

\newcommand{\lsym}{I}

\begin{definition}
  Let $\Rs \coloneqq (\Sigma, S^p, \Dm^a, \gdiv)$ be a spherical
  skeleton and let $\lsym \subseteq \adiv$ be a subset.  We set
  $S_\lsym \coloneqq \{\alpha \in S : \Dm(\alpha) \subseteq \lsym\}$
  and denote by $R_\lsym$ the root system generated by $S_\lsym$.  We
  define
  \begin{align*}
    \Sigma_\lsym &\coloneqq \Sigma \cap \lspan_\Z S_\lsym\text{,}\\
    S^p_\lsym &\coloneqq S^p \cap S_\lsym\text{,}\\
    \Dm^a_\lsym &\coloneqq \{D \in \Dm^a : \text{there exists $\alpha
      \in
      \Sigma_\lsym$ with $D \in \Dm(\alpha)$}\}\text{.}
  \end{align*}
  Then $\Ss_\lsym \coloneqq (\Sigma_\lsym, S^p_\lsym, \Dm^a_\lsym)$ is a spherically
  closed spherical system for the root system $R_\lsym$, and, as usual, we denote the full set
  of colors by $\bdiv_\lsym$. We set $\gdiv_\lsym \coloneqq \lsym \setminus \Dm_\lsym$.
  The spherical skeleton $\Rs_{\lsym} \coloneqq (\Sigma_\lsym,
  S^p_\lsym, \Dm^a_\lsym, \gdiv_\lsym)$ is called the
  \emph{localization} of $\Rs$ at $\lsym$.
\end{definition}

If $X$ is an arbitrary spherical variety and $Y \subseteq X$ is a
$G$-orbit, then the slice $Z$ provided by the local structure theorem
for spherical varieties satisfies the assumptions of
Theorem~\ref{th:smoothness} and has spherical skeleton $\Rs_Z =
(\Rs_{X})_\lsym$ where $\lsym \subseteq \adiv$ is the set of prime
divisors $D$ such that $Y \subseteq \overline{D}$. It follows that
Theorem~\ref{th:smoothness} implies the following result (for details,
we refer to \cite{gcscsv}):

\begin{theorem}
  Let $X$ be a spherical variety, and let $Y \subseteq X$ be a
  $G$-orbit. We denote by $\lsym \subseteq \adiv$ the set of prime
  divisors $D$ such that $Y \subseteq \overline{D}$. If
  Conjecture~\ref{conj:esgp} holds, then $X$ is smooth along $Y$ if
  and only if $\is{\Rs_{\lsym}} = |R_{\lsym}^+ \setminus
  R^+_{S^p_{\lsym}}|$. \hfill \qed
\end{theorem}

We now give the proof of Theorem~\ref{th:smoothness}. Let $X$ be a
locally factorial affine spherical $G$-variety such that the derived
subgroup $[G, G]$ fixes pointwise the unique closed $G$-orbit in $X$.

\begin{lemma}
  The spherical skeleton $\Rs_X$ is complete.
\end{lemma}
\begin{proof}
  As $X$ is affine, it contains exactly one closed
  $G$-orbit $Y$, which corresponds to a
  colored cone $(\Cm, \Fm)$ (see, for instance,
  \cite[Theorem 6.7]{knopsph}).
  By assumption, $[G,G]$ is contained in
  the stabilizer of the open $B$-orbit in $Y$, hence $S_Y^p =
  S$ where $(\Sigma_Y, S^p_Y, \Dm^a_Y)$ is the spherically closed
  spherical system associated to $Y$.
  By \cite[Theorem~1.1]{gh15}, we have $S_Y^p = \{ \alpha \in S :
  \Dm(\alpha) \subseteq \Fm \}$, hence $\Fm = \Dm$, which means that
  $Y$ is contained in the closure of every $B$-invariant prime divisor
  in $X$. It follows that we
  have $\Cm = \cone(\rho'(D) : D \in \adiv) \subseteq \Nm_\Q$.
  We denote by $\overline{\Cm}$
  (resp.~by $\overline{\Vm}$) the image of $\Cm$ (resp.~of $\Vm$) in
  $\Lambda^*_\Q$. Then we have $\overline{\Cm} = \cone(\rho(D):D \in
  \adiv)$. As $\Cm^\circ \cap \Vm \ne \emptyset$
  and $\Cm^\circ$ is mapped to $\overline{\Cm}{}^\circ$,
  we obtain $\overline{\Cm}{}^\circ \cap \overline{\Vm} \ne \emptyset$.
  According to
  \cite[Lemma~2.1.2]{brcox}, we have $-\overline{\Vm} \subseteq \Em
  \coloneqq \cone(\rho(D):D\in\bdiv)$. In particular, the cones $\Em$
  and $\overline{\Cm}$ are full-dimensional in $\Lambda^*_\Q$, so that
  we in particular obtain
  $\overline{\Cm}{}^\circ \cap \overline{\Vm}{}^\circ
  = \topint(\overline{\Cm}) \cap \topint(\overline{\Vm}) \ne
  \emptyset$. Let $u \in \overline{\Cm}{}^\circ \cap
  \overline{\Vm}{}^\circ$. Then we have $u \in -\Em^\circ$, hence
  $\overline{\Cm} \supseteq \cone(\Em \cup\{u\}) = \Lambda^*_\Q$, \ie
  $\Rs_X$ is complete.
\end{proof}

\begin{proof}[Proof of Theorem~\ref{th:smoothness}]
  It follows from \cite[Proposition~6.2]{camus}, see also
  \cite[Proposition~3.3]{gcscsv}, that $X$ is smooth if and only if
  $\Rs_X$ is isomorphic to the spherical skeleton of a
  multiplicity-free space.  According to Conjecture~\ref{conj:esgp2},
  this is equivalent to $\is{X} = \dim X - \rank X$.
\end{proof}

\section{Symmetric varieties: the inequality}
\label{sec:symmetric}

The purpose of the following three sections is to prove
Theorem~\ref{th:sym}. In fact, we prove Conjecture~\ref{conj:esgp2}
for spherical skeletons coming from symmetric varieties.

Let $H \subseteq G$ be a symmetric subgroup. By the definition of a
symmetric subgroup, there exists an involution $\theta \colon G \to G$
with $(G^\theta)^\circ \subseteq H \subseteq N_G(G^\theta)$. According to
\cite[Remarque~2.1]{vust90}, we may assume $G = G^{ss} \times C$ where
$G^{ss}$ is semisimple simply-connected and $C$ is a torus as well as
$\theta(G^{ss}) = G^{ss}$ and $\theta(c)= c^{-1}$ for every $c \in
C$. Then, in particular, $C^\theta = \{e\}$.

\begin{prop}
  \label{prop:symcl}
  The spherical closure $\overline{H}$ is a symmetric subgroup of $G$.
\end{prop}
\begin{proof}
  As $C \subseteq \overline{H}$, $\theta(G^{ss}) = G^{ss}$, and
  $\theta(C) = C$, in order to show that $\overline{H}$ is symmetric,
  we may assume $G=G^{ss}$. By \cite[1.7, Lemma ii)]{cp83}, we obtain
  an exact sequence
  \begin{align*}
    1 \to G^\theta \to N_G(G^\theta) \to \{ g \theta(g)^{-1} : g \in G
    \} \cap Z(G) \to 1.
  \end{align*}
  As $G=G^{ss}$, the center $Z(G)$ is finite, and hence $G^\theta$ has
  finite index in $N_G(G^\theta)$ (see also \cite[1.7, Remark
  a)]{cp83}). Then we have $(N_G(G^\theta))^\circ = G^\theta$ and, as
  $N_G(K^\circ) = N_G(K)$ for every spherical subgroup $K \subseteq
  G$,
  \begin{align*}
    N_G(N_G(G^\theta)) = N_G((N_G(G^\theta))^\circ) =
    N_G(G^\theta)\text{,}
  \end{align*}
  hence $\overline{H} \subseteq \overline{N_G(G^\theta)} \subseteq N_G(N_G(G^\theta)) = N_G(G^\theta)$.
\end{proof}

\begin{cor}
  \label{cor:symprod}
  The quotient $G / \overline{H}$ can be written as a product of
  finitely many factors of the following form:
  \begin{enumerate}
  \item $(G' \times G') / G'$ where $G'$ is a simple adjoint group and
    diagonally embedded in $G'\times G'$.
  \item $G' / H'$ where $G'$ is a simple group and $H'$ is a
    spherically closed symmetric subgroup.
  \end{enumerate}
\end{cor}
\begin{proof}
  As $\overline{H}$ contains the center of $G$, we may assume that $G$
  is semisimple simply-connected.  By Proposition \ref{prop:symcl},
  the spherical closure $\overline{H} \subseteq G$ is symmetric, \ie
  there exists an involution $\theta \colon G \to G$ with $G^\theta
  \subseteq \overline{H} \subseteq N_G(G^\theta)$, as we have
  $(G^\theta)^\circ = G^\theta$ for simply-connected $G$ according to
  \cite[8.2]{st68}.

  The involution $\theta$ either swaps two simple factors of $G$ or
  can be restricted to an involution on a simple factor of $G$.  Hence
  we can write $G = G_1 \times \dots \times G_k$ and $G^\theta =
  G^\theta_1 \times \dots \times G^\theta_k$ where for every $G_i$
  either $G_i$ is simple or $G_i = G'_i \times G'_i$ with $G'_i$
  simple and $G_i^\theta \cong G'_i$ diagonally embedded in $G'_i \times
  G'_i$. We also obtain $N_G(G^\theta) = N_G(G^\theta_1) \times \dots
  \times N_G(G^\theta_k)$.
  
  In particular, the spherically closed spherical system of
  $N_G(G^\theta)$ can be written as a product $\Ss = \Ss_1 \times
  \dots \times \Ss_k$.  As $N_G(G^\theta) = N_G(\overline{H})$, it
  follows from \cite[Lemma~2.4.1]{bp} that the spherical system of
  $\overline{H}$ can also be written as a product $\Ss' = \Ss'_1
  \times \dots \times \Ss'_k$. Hence we also have a product
  decomposition $\overline{H} = \overline{H}_1 \times \dots \times
  \overline{H}_k$.
  
  In the case $G_i = G'_i \times G'_i$ with $G'_i$ simple, we may
  assume that $G'_i$ is adjoint as $\overline{H}_i$ contains the
  center of $G_i$. Then we have $G_i^\theta = G'_i = N_{G_i}(G'_i) =
  N_{G_i}(G_i^\theta)$ (see \cite[Section~2.2, Lemme~1]{vust90}) with
  $G'_i \hookrightarrow G'_i \times G'_i$ diagonally embedded, hence
  $\overline{H}_i = G'_i$.
\end{proof}

\begin{definition}
  \label{def:symm-sph-sys}
  A spherically closed spherical system $\Ss$ is called
  \emph{symmetric} if it is the spherically closed spherical system
  associated to a symmetric subgroup. A spherical skeleton $\Rs$ is
  called \emph{symmetric} if the underlying spherically closed
  spherical system is symmetric.
\end{definition}

\begin{cor}
  \label{cor:sym-sph-sys-prod}
  Every symmetric spherical system $\Ss$ can be written as a product
  of symmetric spherical systems appearing in
  Appendix~\ref{sec:luna-diagr-symm}.
\end{cor}

\begin{definition}
  \label{def:elskel}
  A spherical skeleton $\Rs \coloneqq (\Sigma, S^p, \Dm^a, \gdiv)$ is
  called \emph{elementary} if we have $\langle \rho(\gdiv), \Sigma \rangle
  \subseteq \{0, -1\}$ and for every $D \in \gdiv$ there exists at
  most one $\gamma \in \Sigma$ with $\langle \rho(D), \gamma\rangle =
  -1$.  An elementary spherical skeleton is called \emph{reduced} if
  for every $\gamma \in \Sigma$ there exists at most one $D \in \gdiv$
  with $\langle \rho(D), \gamma\rangle = -1$.

  If $\Rs$ is an arbitrary spherical skeleton, we associate to it an
  elementary spherical skeleton $\Rs^{\el} \coloneqq (\Sigma, S^p,
  \Dm^a, \gdiv^{\el})$ as follows: We write $n_\gamma \coloneqq
  -\sum_{D\in\Gamma} \langle \rho(D), \gamma \rangle$, denote by
  $\gdiv_\gamma$ a set with $n_\gamma$ elements such that for every $D
  \in \gdiv_\gamma$ we have $\langle \rho(D), \Sigma\rangle \subseteq
  \{0, -1\}$ with $\langle \rho(D), \gamma'\rangle = -1$ if and only
  if $\gamma' = \gamma$, and define $\gdiv^{\el}$ to be the disjoint
  union of the $\gdiv_\gamma$ for $\gamma \in \Sigma$.

  We also associate to $\Rs$ a reduced elementary spherical skeleton
  $\Rs^{\vel}$: We start with $\Rs^{\el}$, but replace $\gdiv^{\el}$
  with $\gdiv^{\vel}$ where we retain only one element from each
  $\gdiv_\gamma$.  We write $\|\gdiv\| \subseteq \Sigma$ for the
  subset of those $\gamma$ with $n_\gamma > 0$. For a reduced
  elementary spherical skeleton $\Rs$, the set $\gdiv$ is determined by
  $\|\gdiv\|$.
\end{definition}

\begin{remark}
  Let $\Rs$ be a spherical skeleton. If $\Rs$ is complete, then
  $\Rs^{\el}$ and $\Rs^{\vel}$ are complete.
\end{remark}

\begin{prop}
  \label{prop:elem}
  Let $\Rs \coloneqq (\Sigma, S^p, \Dm^a, \gdiv)$ be a spherical
  skeleton.  Then we have \begin{align*}\is{\Rs} \le \is{\Rs^{\el}}
    \le \is{\Rs^{\vel}}\text{.}\end{align*}
\end{prop}
\begin{proof}
  Let $\gamma \in \Sigma$, $D' \in \gdiv_\gamma \subseteq
  \gdiv^{\el}$, and let $D \in \gdiv$ such that $\langle \rho(D),
  \gamma \rangle \ne 0$.  Writing any $\sv \in \cone(\Sigma)$ as a
  nonnegative linear combination of $\Sigma$, we observe $\langle
  \rho(D), \sv \rangle \le \langle \rho(D'), \sv \rangle$.  Then, for
  $\sv \in \Qm_{\Rs}^* \cap \cone(\Sigma)$, we have
  \begin{align*}
    \langle \rho(D'), \sv \rangle \ge \langle \rho(D), \sv \rangle \ge -m_D
    = -m_{D'} = -1\text{,}
  \end{align*}
  hence $\Qm_{\Rs}^* \cap \cone(\Sigma) \subseteq \Qm_{\Rs^{\el}}^*
  \cap \cone(\Sigma)$. As $\sum_{D\in\Gamma^{\el}} \rho(D) =
  \sum_{D\in\Gamma} \rho(D)$, we obtain $\is{\Rs} \le \is{\Rs^{\el}}$.

  On the other hand, we have $\Qm^*_{\Rs^{\vel}} = \Qm^*_{\Rs^{\el}}$
  and $\sum_{D\in\Gamma^{\el}} \rho(D) \le \sum_{D\in\Gamma^{\vel}}
  \rho(D)$ on $\Sigma$, so that we obtain $\is{\Rs^{\el}} \le
  \is{\Rs^{\vel}}$.
\end{proof}

\begin{remark}
  \label{rem:less-markings}
  Let $\Rs_1 \coloneqq (\Sigma, S^p, \Dm^a, \gdiv_1)$ and $\Rs_2
  \coloneqq (\Sigma, S^p, \Dm^a, \gdiv_2)$ be two reduced elementary
  spherical skeletons with $\|\gdiv_1\| \subseteq \|\gdiv_2\|$. Then
  we have $\is{\Rs_2} \le \is{\Rs_1}$.
\end{remark}

\begin{theorem}
  \label{thm:p-leq}
  Let $\Rs$ be a complete symmetric spherical skeleton. Then we have
  $\is{\Rs} \le |R^+ \setminus R^+_{S^p}|$.
\end{theorem}
\begin{proof}
  According to Proposition~\ref{prop:elem}, it suffices to show
  $\is{\Rs^{\vel}} \le |R^+ \setminus R^+_{S^p}|$, \ie we may assume
  that $\Rs$ is reduced elementary.  If we write $\Ss=\Ss_1 \times
  \ldots \Ss_k$ as product of symmetric spherically closed spherical
  systems $\Ss_i = (\Sigma_i, S^p_i, \Dm^a_i )$ from the list in
  Appendix~\ref{sec:luna-diagr-symm} (see
  Corollary~\ref{cor:sym-sph-sys-prod}), then we may decompose $\Gamma
  = \Gamma_1 \cup \ldots \cup \Gamma_k$ as a disjoint union of subsets
  $\Gamma_i \subseteq \Gamma$ such that $\rho(D) \in \Lambda_i^*$ for
  $D \in \Gamma_i$ where $\Lambda_i$ is the lattice generated by
  $\Sigma_i$. This defines a product $\Rs = \Rs_1 \times \ldots \times
  \Rs_k$. It is straightforward to check that $\is{\Rs} =
  \sum_{i=1}^k\is{\Rs_i}$. As $\Rs$ is complete, it follows that the
  $\Rs_i$ are complete as well.  Hence we may assume that $\Rs$ is a
  complete spherical skeleton with underlying spherically closed
  spherical system from the list in Appendix
  \ref{sec:luna-diagr-symm}.

  As symmetric subgroups are reductive, the image of the colors under
  $\rho$ generate a strictly convex cone in $\Lambda^*_\Q$
  (see~\cite[Theorem~6.7]{knopsph}). As $\Rs$
  is complete, this means that $\|\gdiv\|$ contains at least one
  element.  By Remark \ref{rem:less-markings} it suffices to prove
  $\is{\Rs} \le |R^+ \setminus R^+_{S^p}|$ in the case where
  $\|\gdiv\|$ contains exactly one element.  In
  Appendix~\ref{sec:tables} we have computed all the possible cases,
  and it is then straightforward to check that the assertion is
  true. An example for such a computation can be found in the next
  section.
\end{proof}

\section{Symmetric varieties: solving linear programs}

In this section we explain a method to compute $\is{\Rs}$ where $\Rs
\coloneqq ( \Sigma, S^p, \Dm^a, \Gamma )$ is a reduced elementary symmetric
spherical skeleton with spherically closed spherical system $(\Sigma,
S^p, \Dm^a)$ from the list in Appendix~\ref{sec:luna-diagr-symm} and
$|\Gamma| = 1$.  We will illustrate this method on the example
$\PSL(n{+}1) \subseteq \PSL(n{+}1) \times \PSL(n{+}1)$. The remaining cases
can be handled analogously.

We denote the simple roots of the first (resp.~the second) simple
factor of $\PSL(n{+}1) \times \PSL(n{+}1)$ by $\alpha_1, \ldots,
\alpha_n$ (resp.~by $\alpha'_1, \ldots, \alpha'_n$). Then we have
$\Sigma = \{ \alpha_i + \alpha'_i : i = 1, \dots, n \}$, $S^p =
\emptyset$, $\Dm^a = \emptyset$, and $\| \Gamma \| = \{ \gamma_k \}$
for some $1 \le k \le n$ where $\gamma_k = \alpha_k + \alpha'_k$.

We identify $\Lambda_\Q \otimes_\Q \R \cong \R^\Sigma$ and define $c
\in \R^\Sigma$ by $c_\gamma \coloneqq \sum_{D \in \Delta}
\langle\rho(D), \gamma\rangle$ as well as $b \in \R^\adiv$ by $b_D
\coloneqq 1$ and denote by $A \in \R^{\adiv \times \Sigma}$ the matrix
with $A_{D,\gamma} \coloneqq \langle -\rho(D) / m_D,
\gamma\rangle$. Then we have
\begin{align*}
  \is{\Rs} = \sum_{D\in\adiv} (m_D-1) + \sup{} \rleft\{c^tx :
  \text{$x\in\R^\Sigma$, $Ax \le b$, $x \ge 0$}\rright\}\text{,}
\end{align*}
where the supremum is the optimal value of a linear program. Passing
to the dual linear program, we obtain
\begin{align*}
  \is{\Rs} = \sum_{D\in\adiv} (m_D-1) + \inf{}\rleft\{b^ty : \text{$y
    \in \R^{\adiv}$, $A^ty \ge c$, $y \ge 0$}\rright\}\text{,}
\end{align*}
which is true if either of the linear programs has a finite optimal
value. We are going to solve the dual linear program, in particular
showing that it attains a finite infimum.

The Bourbaki numbering of the simple roots induces a numbering of
$\Sigma$, hence also a numbering of $\Dm$ as we have a natural
bijection bijection $\Dm \cong \Sigma$. We extend this numbering to
$\Delta = \Dm \cup \Gamma$ where the element in $\Gamma$ is given the
number $n+1$. We obtain natural identifications $\R^\Sigma \cong
\R^n$, $\R^\Delta \cong \R^{n+1}$, and $\R^{\Delta \times \Sigma}
\cong \R^{(n+1) \times n}$. We denote the standard basis of $\R^n$ by
$e_1, \ldots, e_n$.

Because of the symmetry of the Luna diagram, we may assume that $k
\le \rleft\lceil n/2 \rright\rceil$ where $\lceil x \rceil$ denotes
the smallest integer greater than or equal to the real number
$x$.

We have $c = e_1 - e_k + e_n$ and the $n \times n$-submatrix of $A$
consisting of the first $n$ rows is $-\frac{1}{2}A_n$ where $A_n$
denotes the Cartan matrix of the root system of
$\PSL(n{+}1)$. Moreover, the last row in $A$ coincides with $e_k \in
\R^n$.

\begin{prop}
  \label{prop:an_lower_bounds}
  For every $y \in \R^{n+1}_{\ge0}$ subject to $A^ty \ge c$ we have:
  \begin{align*}
    y_i \ge
    \begin{cases}
      2(i-1) + i y_1 & \text{for } 1 \le i \le k\text{,}\\
      2(n-i) + (n-i+1) y_n & \text{for } k \le i \le n\text{,}\\
      y_k - \tfrac{1}{2} ( y_{k-1} + y_{k+1} ) - 1 & \text{for } i = n
      + 1\text{.}
    \end{cases}
  \end{align*}
\end{prop}
\begin{proof}
  First, we verify the assertion for $1 \le i \le k$. The cases $i =
  1, 2$ are obvious. For $2 < j \le k$ it follows from $A^t y \ge c$
  that
  \begin{equation}
    \label{eq:a}
    y_j \ge 2 y_{j-1} - y_{j-2}\text{.}\tag{$\ast$}
  \end{equation}
  Let $i > 2$ and consider \eqref{eq:a} for $j \coloneqq i$. Then
  successively apply \eqref{eq:a} for $j\coloneqq i-1, \dots, 3$,
  resulting in the inequality $y_i \ge ( i - 1 ) y_2 - ( i - 2 )
  y_1$. Further applying the inequality $y_2 \ge 2 + 2 y_1$, we obtain
  $y_i \ge 2( i - 1 ) + i y_1$.

  Next, we verify the assertion for $k \le i \le n$. The cases $i =
  n-1, n$ are obvious. For $k \le j < n-1$ it follows from $A^t y \ge
  c$ that
  \begin{equation}
    \label{eq:b}
    y_j \ge 2 y_{j+1} - y_{j+2}\text{.}\tag{$\ast\ast$}
  \end{equation}
  Let $i < n-1$ and consider \eqref{eq:b} for $j \coloneqq i$. Then
  successively apply \eqref{eq:b} for $j \coloneqq i+1, \dots, n-2$,
  resulting in the inequality $y_i \ge (n-i) y_{n-1} - (n-i-1)y_n$.
  Further applying the inequality $y_{n-1} \ge 2 + 2 y_n$, we obtain
  $y_i \ge 2 (n-i) + ( n - i + 1 ) y_n$.
\end{proof}

\begin{cor}
  The dual linear program achieves its minimum at the point $y \in
  \R^{n+1}$ with
  \begin{align*}
    y_i =
    \begin{cases}
      2(i-1) & \text{for} \; i = 1, \ldots, k-1 \\
      2(n-i) & \text{for} \; i = k, \ldots, n \\
      n - 2(k-1) & \text{for} \; i = n+1.
    \end{cases}
  \end{align*}
  and this minimum is given by $n^2-2kn+3n+2k^2-6k+4$.
\end{cor}
\begin{proof}
  As we can find a solution $x$ of the original linear program which
  is a vertex of $\Qm^*_{\Rs} \cap \cone(\Sigma)$, by the strong
  duality theorem, the dual linear program also has an optimal
  solution which we denote by $y'$. By Proposition
  \ref{prop:an_lower_bounds}, we obtain
  \begin{align*}
    \sum_{i=1}^{n+1} y'_i & \ge \sum_{i=1}^{k-2} y'_i + \tfrac{1}{2}
    y'_{k-1} + 2 y'_k + \tfrac{1}{2} y'_{k+1} + \sum_{i=k+2}^n y'_i -
    1 \\
    & \ge n^2-2kn+2n+2k^2-6k+4\text{.}
  \end{align*}
  Observe that because of the assumption $k \le \lceil n/2 \rceil$ we
  have $y_{n+1} \ge 0$. It is straightforward to verify $y \ge 0$,
  $A^ty = c$, and
  \begin{align*}
    \sum_{i=1}^{n+1} y_i = n^2-2kn+2n+2k^2-6k+4\text{.}
  \end{align*}
  It follows that we have
  \begin{align*}
    \inf{}\rleft\{b^ty' : \text{$y' \in \R^{n+1}$, $A^ty' \ge c$, $y'
      \ge 0$}\rright\} = n^2-2kn+2n+2k^2-6k+4\text{,}
  \end{align*}
  and we obtain $\is{\Rs} = n^2-2kn+3n+2k^2-6k+4$.
\end{proof}

\section{Symmetric varieties: the case of equality}

\begin{definition}
  A spherical skeleton $\Rs$ is called \emph{linear} if $\Rs \cong
  \Rs_V$ for some multiplicity-free space $V$.  A linear spherical
  skeleton $\Rs \cong \Rs_V$ is called \emph{indecomposable} if the
  multiplicity-free space $V$ is indecomposable (see
  \cite[Section~1]{gcscsv} for the definition of indecomposable
  multiplicity-free spaces). The indecomposable linear spherical
  skeletons correspond to 
  the marked Luna diagrams listed in \cite[Section~2]{gcscsv}.
\end{definition}

\begin{remark}
  In Appendix~\ref{sec:tables} we have computed $\is{\Rs}$ for all
  reduced elementary $\Rs$ with $\Ss$ from the list in
  Appendix~\ref{sec:luna-diagr-symm} and $|\gdiv| = 1$. From these we
  have gathered in Table~\ref{tab:upper-bound-achieved} those cases
  where $\is{\Rs} = |R^+ \setminus R^+_{S^p}|$.  As all of these cases
  appear in \cite[Section~2]{gcscsv}, they are (indecomposable)
  linear.  Moreover, in these cases the face of the polytope
  $\Qm_{\Rs}^*$ where the optimal value of $\is{\Rs}$ is achieved is a
  vertex $\sv \in \Qm_{\Rs}^* \cap \cone(\Sigma)$,
  which we also give in Table~\ref{tab:upper-bound-achieved}
  (in the last column).
  \begin{table}[!ht]
    \centering
    \renewcommand{\arraystretch}{1.1}
    \begin{tabular}{@{}llll@{}}
      \toprule
      No. & conditions & $\|\Gamma\|$ & vertex $\sv$\\
      \midrule
      $2$ & $G$ of type $A_n$ & $\gamma_1$ & $\sum_{k=1}^n k^2 \gamma_k$ \\
      $2$ & $G$ of type $A_n$ & $\gamma_n$ & $\sum_{k=1}^n k^2 \gamma_{n-k}$ \\
      $3$ & $l {\ge} 1$, $m{=}0$ & $\gamma_1$ & $\gamma_1$ \\
      $4$ & $m{=}0$ & $\gamma_1$ & $\gamma_1$ \\
      $5$ & & $\gamma_1$ & $\tfrac{1}{2} \sum_{k=1}^m (k^2{+}k) \gamma_k$ \\
      $5$ & & $\gamma_m$ & $\tfrac{1}{2} \sum_{k=1}^m (k^2{+}k) \gamma_{m-k}$ \\
      $6$ & & $\gamma_1$ & $\sum_{k=1}^m (2k^2 {-} k) \gamma_k$ \\
      $6$ & & $\gamma_m$ & $\sum_{k=1}^m (2k^2 {-} k) \gamma_{m-k}$ \\
      $9$ & $l {\ge} 2$, $m {=} 0$ & $\gamma_1$ & $\gamma_1$ \\
      $15$ & $l {\ge} 3$, $m {=} 0$ & $\gamma_1$ & $\gamma_1$ \\
      $19$ & & $2\alpha_1 {+} 2\alpha_3 {+} 2\alpha_4 {+} \alpha_2 {+} \alpha_5$ & $\gamma_1 {+} 10 \gamma_2$\\
      $19$ & & $2\alpha_6 {+} 2\alpha_5 {+} 2\alpha_4 {+} \alpha_2 {+} \alpha_3$ & $10 \gamma_1 {+} \gamma_2$\\
      \bottomrule
    \end{tabular}
    \caption{cases where the upper bound is achieved}
    \label{tab:upper-bound-achieved}
  \end{table}
\end{remark}

Let $\Rs \coloneqq (\Sigma, S^p, \Dm^a, \Gamma)$ be a complete
symmetric spherical skeleton with underlying spherically closed
spherical system $\Ss \coloneqq (\Sigma, S^p, \Dm^a)$ and $\is{\Rs} =
|R^+\setminus R^+_{S^p}|$.
Without loss of generality, we may assume $\Rs = [\Rs]$.
 According to Proposition~\ref{prop:elem}
and Theorem~\ref{thm:p-leq}, we have $\is{\Rs}=\is{\Rs^{\vel}} = |R^+
\setminus R^+_{S^p}|$. If we write $\Ss=\Ss_1 \times \ldots \times
\Ss_k$ as product of spherically closed spherical systems $\Ss_i = (
\Sigma_i, S^p_i, \Dm^a_i )$ from the list in
Appendix~\ref{sec:luna-diagr-symm} with underlying root systems $R_i$
(see Corollary \ref{cor:sym-sph-sys-prod}), we may accordingly
decompose $\Gamma^{\vel} = \Gamma^{\vel}_1 \cup \dots \cup
\Gamma^{\vel}_k$ such that $\rho(D) \in \Lambda_i^*$ for $D \in
\Gamma^{\vel}_i$ where $\Lambda_i$ is the lattice generated by
$\Sigma_i$. This defines a product $\Rs^{\vel} = \Rs^{\vel}_1 \times
\ldots \times \Rs^{\vel}_k$.  It is straightforward to verify
$\is{\Rs^{\vel}} = \sum_{i=1}^k\is{\Rs^{\vel}_i}=|R^+ \setminus
R^+_{S^p}|=\sum_{i=1}^k|R_i^+ \setminus R_{i,S_i^p}^+|$, hence
$\is{\Rs_i^{\vel}}=|R_i^+ \setminus R_{i,S_i^p}^+|$.

\begin{lemma}
  \label{lem:ee-one}
  $\|\Gamma_i^{\vel}\|$ contains exactly one element.
\end{lemma}
\begin{proof}
  We denote by $\widetilde{\Sigma}$ the set of spherically closed
  spherical roots appearing in Table~\ref{tab:upper-bound-achieved}.

  Let $\gamma \in \|\Gamma^{\vel}_i\|$ and denote by $\Rs'$
  the spherical skeleton
  obtained from $\Rs_i^{\vel}$ by removing every element from
  $\|\Gamma_i^{\vel}\|$ except $\gamma$.
  By Remark~\ref{rem:less-markings}, we have $\is{\Rs^{\vel}_i} \le \is{\Rs'}$.
  As we have $\is{\Rs'} = |R_i^+ \setminus R_{i,S_i^p}^+|$, an inspection
  of Appendix~\ref{sec:tables} yields $\gamma \in \widetilde{\Sigma}$.

  According to
  Table~\ref{tab:upper-bound-achieved}, $\|\gdiv_i^{\vel}\|$ can only
  contain one or two elements (two could be possible for No. 2, No. 5,
  No. 6, and No. 19). However, in the cases with two elements it can be
  readily verified that $\is{\Rs_i^{\vel}} < |R_i^+\setminus
  R^+_{i,S_i^p}|$.
\end{proof}

\begin{prop}
  \label{prop:vellin}
  $\Rs_i^{\vel}$ is linear indecomposable. In particular, $\Rs^{\vel}$
  is linear.
\end{prop}
\begin{proof}
  By Lemma~\ref{lem:ee-one}, we obtain $|\Gamma_i^{\vel}|=1$, which
  means that $\Rs_{i}^{\vel}$ appears in
  Table~\ref{tab:upper-bound-achieved}.
\end{proof}

\begin{prop}
  \label{prop:elvel}
  We have $\Rs^{\el} = \Rs^{\vel}$.
\end{prop}
\begin{proof}
  It suffices to show that if $\Rs$ is indecomposable linear, then
  replacing the unique element $D \in \gdiv$ by two distinct elements
  $D', D''$ with $\rho(D) = \rho(D') = \rho(D'')$ makes $\is{\Rs}$
  strictly smaller. Let us denote this new spherical skeleton by
  $\Rs'$. The polytope $\Qm_{\Rs}^*$ does not change, \ie $\Qm_{\Rs}^*
  = \Qm_{\Rs'}^*$, while $\is{\Rs'} = \is{\Rs} + \langle \rho(D), \sv
  \rangle$ where $\sv \in \Qm_{\Rs}^* \cap \cone(\Sigma)$ is the
  vertex in Table~\ref{tab:upper-bound-achieved}. By
  the explicit description of this vertex,
  it follows that $\langle \rho(D), \sv \rangle < 0$.
\end{proof}

\begin{lemma}
  \label{lemma:split}
  Let $\Rs_1 \coloneqq (\Sigma, S^p, \Dm^a, \Gamma_1)$ and $\Rs_2
  \coloneqq (\Sigma, S^p, \Dm^a, \Gamma_2)$ be spherical skeletons
  with $\Rs_1^{\vel} = \Rs_1^{\el}$ and $\Rs_2^{\vel} =
  \Rs_2^{\el}$. Assume that there exists a surjective map $\phi\colon
  \gdiv_2 \to \gdiv_1$ such that for every $D \in \gdiv_1$ we have
  $\sum_{D'\in\phi^{-1}(D)} \rho_2(D') = \rho_1(D)$.  Then we have
  $\is{\Rs_1} \le \is{\Rs_2}$.
\end{lemma}
\begin{proof}
  As in the proof of Proposition~\ref{prop:elem}, we obtain
  $\Qm^*_{\Rs_1} \cap \cone(\Sigma) \subseteq \Qm^*_{\Rs_2} \cap
  \cone(\Sigma)$ and the linear function to be maximized is the same
  for both spherical skeletons.
\end{proof}

\begin{remark}
  \label{rem:seq}
  In the situation of Lemma~\ref{lemma:split}, we write $\Rs_1
  \trianglelefteq \Rs_2$. 
  It is clear that, if $\Rs^{\el}
  \ne \Rs$, then we can always find a sequence
  \begin{align*}
    \Rs \trianglelefteq \Rs_1 \trianglelefteq \dots \trianglelefteq
    \Rs_l \trianglelefteq \Rs^{\el}
  \end{align*}
  such that (after possibly renumbering $\Ss_1, \dots, \Ss_k$) there
  exists an index $i$ with $\Rs_{i} = \Rs'_{i} \times \Rs''$,
  $\Rs_{i+1} = \Rs'_{i+1} \times \Rs''$ such that $\Rs'_i$ and
  $\Rs'_{i+1}$ have underlying spherically closed spherical system
  $\Ss_1 \times \Ss_2$ with $\Rs'_i \trianglelefteq \Rs'_{i+1} =
  (\Rs'_{i+1})^{\el}$ and $\Rs'_i \ne \Rs'_{i+1}$.
\end{remark}

\begin{lemma}
  \label{lemma:lu}
  Assume $\Ss = \Ss_1 \times \Ss_2$ and $|\gdiv| = 1$. Then we have
  $\is{\Rs} < \is{\Rs^{\el}}$.
\end{lemma}
\begin{proof}
  The disjoint union $\Sigma = \Sigma_1 \cup \Sigma_2$ yields a direct
  sum decomposition $\Lambda = \Lambda_1 \oplus \Lambda_2$. Moreover,
  it is straightforward to show $\Qm_{\Rs^{\vel}}^* =
  \Qm_{\Rs_1^{\vel}}^* \times \Qm_{\Rs_2^{\vel}}^*$. We have seen that
  $\is{\Rs_i^{\vel}}$ achieves its optimal value (only) at the vertex
  $\sv_i$ of $\Qm_{\Rs_i^{\vel}}^*$ listed in
  Table~\ref{tab:upper-bound-achieved}.  The result follows from the
  straightforward observation $(\sv_1,\sv_2) \not \in \Qm_{\Rs}^* \cap \cone(\Sigma)
  \subseteq \Qm^*_{\Rs^{\vel}} = \Qm_{\Rs_1^{\vel}}^* \times
  \Qm_{\Rs_2^{\vel}}^*$.
\end{proof}

\begin{theorem}
  $\Rs$ is linear.
\end{theorem}
\begin{proof}
  As $\is{\Rs} = |R^+ \setminus R^+_{S^p}|$, we must have $\is{\Rs_i}
  = \is{\Rs_{i+1}}$ for any sequence as in Remark~\ref{rem:seq}.  It
  then follows from Lemma~\ref{lemma:lu} that we have $\Rs =
  \Rs^{\el}$, which is linear (according to
  Proposition~\ref{prop:vellin} and Proposition~\ref{prop:elvel}).
\end{proof}

\appendix

\section{Luna diagrams of symmetric subgroups}
\label{sec:luna-diagr-symm}

In this appendix, we give the list of symmetric spherically closed
spherical systems for Corollary~\ref{cor:sym-sph-sys-prod}. This list
has been taken from \cite{bp:sph-sys}. For completeness, we have also
explicitly added the Luna diagrams of the group embeddings (from
\cite{br13}).

We will use the following notation: For two natural numbers $p$, $q$
let $S(\GL(p) \times \GL(q)) \subseteq \SL(p+q)$ denote the subgroup
of unimodular block-diagonal matrices with block sizes $p$ and $q$.

\begin{description}
\item[2:] $G \subseteq G\times G$ for every simple adjoint group
  $G$ diagonally embedded in $G\times G$
  \begin{itemize}
  \item $A_n$
    \[
    \begin{picture}(7200,3300)(-300,-300)
      \multiput(0,0)(0,2700){2}{\multiput(0,0)(5400,0){2}{\put(0,0){\usebox{\edge}}
          \multiput(0,0)(1800,0){2}{\circle{600}}}\put(1800,0){\usebox{\susp}}}
      \multiput(0,0)(5400,0){2}{\multiput(0,300)(1800,0){2}{\line(0,1){2100}}}
    \end{picture}
    \]
  \item $B_n$
    \[
    \begin{picture}(9600,3300)(-300,-300)
      \multiput(0,0)(0,2700){2}{\multiput(0,0)(5400,0){2}{\usebox{\edge}}
        \put(7200,0){\usebox{\rightbiedge}}\put(1800,0){\usebox{\susp}}
        \multiput(0,0)(1800,0){2}{\circle{600}}
        \multiput(5400,0)(1800,0){3}{\circle{600}}}
      \multiput(0,300)(1800,0){2}{\line(0,1){2100}}
      \multiput(5400,300)(1800,0){3}{\line(0,1){2100}}
    \end{picture}
    \]
  \item $C_n$
    \[
    \begin{picture}(9600,3300)(-300,-300)
      \multiput(0,0)(0,2700){2}{\multiput(0,0)(5400,0){2}{\usebox{\edge}}
        \put(7200,0){\usebox{\leftbiedge}}\put(1800,0){\usebox{\susp}}
        \multiput(0,0)(1800,0){2}{\circle{600}}
        \multiput(5400,0)(1800,0){3}{\circle{600}}}
      \multiput(0,300)(1800,0){2}{\line(0,1){2100}}\multiput(5400,300)(1800,0){3}{\line(0,1){2100}}
    \end{picture}
    \]
  \item $D_n$
    \[
    \begin{picture}(9900,5700)(-300,-3900)
      \multiput(0,0)(900,-2700){2}{\multiput(0,0)(5400,0){2}{\usebox{\edge}}
        \put(1800,0){\usebox{\susp}}\put(7200,0){\usebox{\bifurc}}
        \multiput(0,0)(5400,0){2}{\multiput(0,0)(1800,0){2}{\circle{600}}}
        \multiput(8400,-1200)(0,2400){2}{\circle{600}}}
      \multiput(100,-300)(5400,0){2}{\multiput(0,0)(1800,0){2}{\line(1,-3){700}}}
      \multiput(8500,-1500)(0,2400){2}{\line(1,-3){700}}
    \end{picture}
    \]
  \item $E_6$, $E_7$, or $E_8$
    \[
    \begin{picture}(12450,5550)(-1350,-2100)
      \multiput(0,0)(-1050,3150){2}{\multiput(0,0)(1800,0){3}{\usebox{\edge}}
        \put(3600,0){\usebox{\vedge}}\put(5400,0){\usebox{\susp}}
        \put(9000,0){\usebox{\edge}}\multiput(0,0)(1800,0){4}{\circle{600}}
        \multiput(9000,0)(1800,0){2}{\circle{600}}
        \put(3600,-1800){\circle{600}}}
      \put(-100,300){\multiput(0,0)(1800,0){4}{\line(-1,3){850}}
        \multiput(9000,0)(1800,0){2}{\line(-1,3){850}}
        \put(3600,-1800){\line(-1,3){850}}}
    \end{picture}
    \]
  \item $F_4$
    \[
    \begin{picture}(6000,3300)(-300,-300)
      \multiput(0,0)(0,2700){2}{\put(0,0){\usebox{\dynkinf}}
        \multiput(0,0)(1800,0){4}{\circle{600}}}
      \multiput(0,300)(1800,0){4}{\line(0,1){2100}}
    \end{picture}
    \]
  \item $G_2$
    \[
    \begin{picture}(6900,2100)(-300,-1200)
      \multiput(0,0)(4500,0){2}{\put(0,0){\usebox{\dynkingtwo}}
        \multiput(0,0)(1800,0){2}{\circle{600}}}
      \multiput(0,-1200)(4500,0){2}{\line(0,1){900}}
      \put(0,-1200){\line(1,0){4500}}
      \multiput(1800,-900)(4500,0){2}{\line(0,1){600}}
      \put(1800,-900){\line(1,0){2600}}\put(4600,-900){\line(1,0){1700}}
    \end{picture}
    \]
  \end{itemize}
\item[3:] $\mathrm S(\GL(m{+}1)\times\GL(m{+}l)) \subseteq
  \SL(2m{+}l{+}1)$,  $|\Sigma|=m{+}1$
  \begin{itemize}
  \item $l=1$, $m=0$
    \[
    \begin{picture}(600,1800)
      \put(300,900){\usebox{\aone}}
    \end{picture}
    \]
  \item $l \ge 2$, $m=0$
    \[
    \begin{picture}(6000,1800)
      \put(300,900){\usebox{\mediumam}}
    \end{picture}
    \]
  \item $l=1$, $m \ge 1$
    \[
    \begin{picture}(11400,2400)(-300,-1500)
      \multiput(0,0)(7200,0){2}{\put(0,0){\usebox{\susp}}
        \multiput(0,0)(3600,0){2}{\circle{600}}}
      \multiput(0,-300)(10800,0){2}{\line(0,-1){1200}}
      \put(0,-1500){\line(1,0){10800}}
      \multiput(3600,-300)(3600,0){2}{\line(0,-1){900}}
      \put(3600,-1200){\line(1,0){3600}}
      \multiput(3600,0)(1800,0){2}{\usebox{\edge}}
      \put(5400,0){\usebox{\aone}}
      \put(5400,600){\usebox{\tow}}
    \end{picture}
    \]
  \item $m \ge 1,$ $l \geq 2$
    \[
    \diagramaapplusqplusp
    \]
  \end{itemize}

\item[4:] $N(\mathrm S(\GL(m{+}1) \times \GL(m{+}1)))
  \subseteq \SL(2m{+}2)$
  \begin{itemize}
  \item $m=0$
    \[
    \begin{picture}(600,1800)
      \put(300,900){\usebox{\aprime}}
    \end{picture}
    \]
  \item $m \geq 1$
    \[
    \diagramaaprimepplusoneplusp
    \]
  \end{itemize}
 
\item[5:] $\SO(m{+}1) \cdot Z_{\SL(m{+}1)}
  \subseteq \SL(m{+}1)$, $m \geq 2$
  \[
  \diagramaon
  \]
 
\item[6:] $\Sp(2m{+}2) \cdot Z_{\SL(2m{+}2)}
  \subseteq \SL(2m{+}2)$, $m \ge 1$
  \[
  \begin{picture}(10800,1800)
    \put(0,600){\diagramacn}
  \end{picture}
  \]

\item[8:] $\SO(2) \times \SO(2l{+}1) \subseteq \SO(2l{+}3)$
  \begin{itemize}
  \item $l=1$
    \[
    \begin{picture}(2400,1800)(-300,-900)
      \put(0,0){\usebox{\rightbiedge}}
      \put(0,0){\usebox{\aone}}
      \put(1800,0){\usebox{\aprime}}
      \put(0,600){\usebox{\toe}}
    \end{picture}
    \]
  \item $l \ge 2$
    \[
    \begin{picture}(9300,1800)(-300,-900)
      \put(0,0){\usebox{\aone}}
      \put(0,0){\usebox{\edge}}
      \put(1800,0){\usebox{\shortbm}}
      \put(0,600){\usebox{\toe}}
    \end{picture}
    \]
  \end{itemize}

\item[9:] $S(\mathrm{O}(m{+}1) \times \mathrm{O}(m{+}2l)) \subseteq
  \SO(2m{+}2l{+}1)$, $|\Sigma| = m{+}1$
  \begin{itemize}
  \item $l \ge 2$, $m=0$
  \[
  \begin{picture}(7500,1800)
    \put(300,900){\usebox{\shortbm}}
  \end{picture}
  \]
  \item $l=m=1$
    \[
    \begin{picture}(2400,1800)(-300,-900)
      \put(0,0){\usebox{\rightbiedge}}
      \multiput(0,0)(1800,0){2}{\usebox{\aprime}}
    \end{picture}
    \]
  \item $l \ge 2$, $m=1$
    \[
    \begin{picture}(9300,1800)(-300,-900)
      \put(0,0){\usebox{\aprime}}
      \put(0,0){\usebox{\edge}}
      \put(1800,0){\usebox{\shortbm}}
    \end{picture}
    \]
 \item $l=1$, $m \ge 2$
    \[
    \begin{picture}(9600,1800)(-300,-900)
      \multiput(0,0)(5400,0){2}{\usebox{\edge}}
      \put(1800,0){\usebox{\susp}}
      \put(7200,0){\usebox{\rightbiedge}}
      \multiput(0,0)(1800,0){2}{\usebox{\aprime}}
      \multiput(5400,0)(1800,0){3}{\usebox{\aprime}}
    \end{picture}
    \]
   \item $l \ge 2$, $m \ge 2$
    \[
    \begin{picture}(14700,1800)(-300,-900)
      \multiput(0,0)(5400,0){2}{\usebox{\edge}}
      \put(1800,0){\usebox{\susp}}
      \multiput(0,0)(1800,0){2}{\usebox{\aprime}}
      \put(5400,0){\usebox{\aprime}}
      \put(7200,0){\usebox{\shortbm}}
    \end{picture}
    \]
  \end{itemize}

\item[10, 11:] $N(\Sp(2m{+}2) \times \Sp(2m{+}2l)) \subseteq \Sp(4m{+}2l{+}2)$, $|\Sigma| = m{+}1$
  \begin{itemize}
  \item $l \ge 2$, $m=0$
    \[
    \begin{picture}(9000,1800)
      \put(0,900){\usebox{\shortcm}}
    \end{picture}
    \]
  \item $l=1$, $m \ge 1$
    \[
    \begin{picture}(12900,1800)(0,-900)
      \put(0,0){\usebox{\dthree}}
      \put(3600,0){\usebox{\susp}}
      \put(7200,0){\usebox{\dthree}}
      \put(10800,0){\usebox{\leftbiedge}}
      \put(12600,0){\usebox{\gcircle}}
    \end{picture}
    \]
  \item $l \ge 2$, $m \ge 1$
    \[
    \begin{picture}(19800,1800)(0,-900)
      \put(0,0){\usebox{\dthree}}
      \put(3600,0){\usebox{\susp}}
      \put(7200,0){\usebox{\dthree}}
      \put(10800,0){\usebox{\shortcm}}
    \end{picture}
    \]
  \end{itemize}

\item[12:] $\GL(m{+}1) \subseteq \Sp(2m{+}2)$, $m \ge 2$
  \[
  \begin{picture}(9600,1800)(-300,-900)
    \multiput(0,0)(5400,0){2}{\usebox{\edge}}
    \put(1800,0){\usebox{\susp}}
    \multiput(0,0)(1800,0){2}{\usebox{\aprime}}
    \multiput(5400,0)(1800,0){2}{\usebox{\aprime}}
    \put(7200,0){\usebox{\leftbiedge}}
    \put(9000,0){\usebox{\aone}}
    \put(9000,600){\usebox{\tow}}
  \end{picture}
  \]

\item[13:] $N(\GL(m{+}1)) \subseteq \Sp(2m{+}2)$, $m \ge 2$
  \[
  \diagramcoprimen
  \]

\item[14:] $\SO(2) \times \SO(2l{+}2) \subseteq
  \SO(2l{+}4)$, $l \ge 2$
  \[
  \begin{picture}(8700,2400)(-300,-1200)
    \put(0,0){\usebox{\aone}}
    \put(0,0){\usebox{\edge}}
    \put(1800,0){\usebox{\shortdm}}
    \put(0,600){\usebox{\toe}}
  \end{picture}
  \]

\item[15:] $S(\mathrm{O}(m{+}1) \times \mathrm{O}(m{+}2l{+}1) ) \subseteq
  \SO(2m{+}2l{+}2)$, $|\Sigma| = m{+}1$
  \begin{itemize}
  \item $l \ge 3$, $m=0$
  
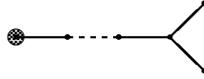
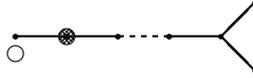
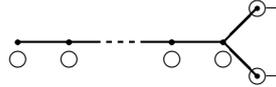
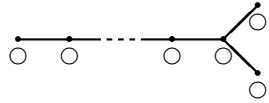
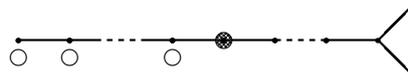
\[
  \begin{picture}(6900,2400)
    \put(300,1200){\usebox{\shortdm}}
  \end{picture}
  \]
  \item $l \ge 2$, $m=1$
    \[
    \begin{picture}(8700,2400)(-300,-1200)
      \put(0,0){\usebox{\aprime}}
      \put(0,0){\usebox{\edge}}
      \put(1800,0){\usebox{\shortdm}}
    \end{picture}
    \]
  \item $l=1$, $m \ge 2$
    \[
    \begin{picture}(9450,3000)(-300,-1500)
      \multiput(0,0)(5400,0){2}{\usebox{\edge}}
      \put(1800,0){\usebox{\susp}}
      \put(7200,0){\usebox{\bifurc}}
      \multiput(0,0)(1800,0){2}{\usebox{\aprime}}
      \multiput(5400,0)(1800,0){2}{\usebox{\aprime}}
      \multiput(8400,-1200)(0,2400){2}{\usebox{\wcircle}}
      \multiput(8700,-1200)(0,2400){2}{\line(1,0){450}}
      \put(9150,-1200){\line(0,1){2400}}
    \end{picture}
    \]
  \item $l=0$, $m \ge 3$
    \[
    \begin{picture}(9000,3600)(-300,-2100)
      \multiput(0,0)(5400,0){2}{\usebox{\edge}}
      \put(1800,0){\usebox{\susp}}
      \put(7200,0){\usebox{\bifurc}}
      \multiput(0,0)(1800,0){2}{\usebox{\aprime}}
      \multiput(5400,0)(1800,0){2}{\usebox{\aprime}}
      \multiput(8400,-1200)(0,2400){2}{\usebox{\aprime}}
    \end{picture}
    \]
    \item  $l, m \ge 2$
    \[
    \diagramdopplusq
    \]
  \end{itemize}

\item[16/1:] $\GL(2m{+}3) \subseteq \SO(4m{+}6)$, $m \ge 1$
  \[
  \diagramdcnodd
  \]

\item[16/2:] $\GL(2m{+}2) \subseteq \SO(4m{+}4)$, $m \ge 1$
  
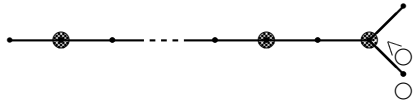
\[
  \begin{picture}(14100,3300)(0,-2100)
    \multiput(0,0)(7200,0){2}{\usebox{\dthree}}
    \put(3600,0){\usebox{\susp}}
    \put(10800,0){\usebox{\edge}}
    \put(12600,0){\usebox{\bifurc}}
    \put(12600,0){\usebox{\gcircle}}
    \put(13800,-1200){\usebox{\aone}}
    \put(13800,-600){\usebox{\tonw}}
  \end{picture}
  \]

\item[17:] $N(\GL(2m{+}2)) \subseteq \SO(4m{+}4)$, $m \ge 1$
  \[
  \diagramdcprimeneven
  \]

\item[18:] $\mathrm D_5 \subseteq \mathrm E_6$
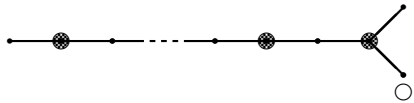
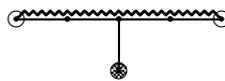
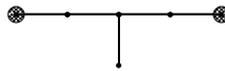
  \[
  \begin{picture}(7800,3000)(-300,-2100)
    \put(0,0){\usebox{\afive}}
    \put(3600,0){\usebox{\vedge}}
    \put(3600,-1800){\usebox{\gcircle}}
  \end{picture}
  \]

\item[19:] $\mathrm F_4 \subseteq \mathrm E_6$
  \[
  \begin{picture}(7800,2700)(-300,-1800)
    \multiput(0,0)(1800,0){4}{\usebox{\edge}}
    \put(3600,0){\usebox{\vedge}}
    \multiput(0,0)(7200,0){2}{\usebox{\gcircle}}
  \end{picture}
  \]

\item[20:] $\mathrm A_5 \times \mathrm A_1 \subseteq \mathrm E_6$
  
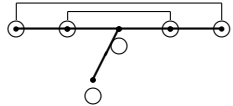
\[
  \begin{picture}(7800,3600)(-300,-2700)
    \put(0,0){\usebox{\dynkinafive}}
    \thicklines\put(3600,0){\line(-1,-2){900}}
    \thinlines\put(2700,-1800){\circle*{150}}
    \multiput(0,0)(5400,0){2}{\multiput(0,0)(1800,0){2}{\circle{600}}}
    \multiput(0,300)(7200,0){2}{\line(0,1){600}}\put(0,900){\line(1,0){7200}}
    \multiput(1800,300)(3600,0){2}{\line(0,1){300}}
    \put(1800,600){\line(1,0){3600}}
    \multiput(3600,0)(-900,-1800){2}{\usebox{\aprime}}
  \end{picture}
  \]

\item[21:] $\mathrm C_4 \subseteq \mathrm E_6$
  
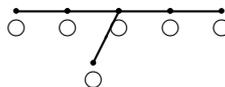
\[
  \begin{picture}(7800,3600)(-300,-2700)
    \put(0,0){\usebox{\dynkinafive}}
    \thicklines
    \put(3600,0){\line(-1,-2){900}}
    \thinlines
    \put(2700,-1800){\circle*{150}}
    \multiput(0,0)(1800,0){5}{\usebox{\aprime}}
    \put(2700,-1800){\usebox{\aprime}}
  \end{picture}
  \]

\item[22:] $\mathrm E_6 \subseteq \mathrm E_7$
  
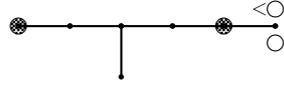
\[
  \begin{picture}(9600,2700)(-300,-1800)
    \multiput(0,0)(1800,0){5}{\usebox{\edge}}
    \put(3600,0){\usebox{\vedge}}
    \multiput(0,0)(7200,0){2}{\usebox{\gcircle}}
    \put(9000,0){\usebox{\aone}}
    \put(9000,600){\usebox{\tow}}
  \end{picture}
  \]

\item[23:] $\mathrm E_6 \subseteq \mathrm E_7$
  
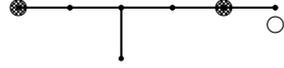
\[
  \begin{picture}(9600,2700)(-300,-1800)
    \multiput(0,0)(1800,0){5}{\usebox{\edge}}
    \put(3600,0){\usebox{\vedge}}
    \multiput(0,0)(7200,0){2}{\usebox{\gcircle}}
    \put(9000,0){\usebox{\aprime}}
  \end{picture}
  \]

\item[24:] $\mathrm D_6 \times \mathrm A_1 \subseteq \mathrm
  E_7$
  
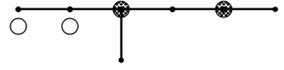
\[
  \begin{picture}(9600,2700)(-300,-1800)
    \multiput(0,0)(1800,0){5}{\usebox{\edge}}
    \put(3600,0){\usebox{\vedge}}
    \multiput(3600,0)(3600,0){2}{\usebox{\gcircle}}
    \multiput(0,0)(1800,0){2}{\usebox{\aprime}}
  \end{picture}
  \]

\item[25:] $\mathrm A_7 \subseteq \mathrm E_7$
  
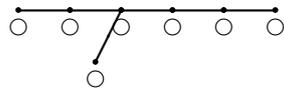
\[
  \begin{picture}(9600,3600)(-300,-2700)
    \put(0,0){\usebox{\dynkinasix}}
    \thicklines
    \put(3600,0){\line(-1,-2){900}}
    \thinlines
    \put(2700,-1800){\circle*{150}}
    \multiput(0,0)(1800,0){6}{\usebox{\aprime}}
    \put(2700,-1800){\usebox{\aprime}}
  \end{picture}
  \]

\item[26:] $\mathrm E_7 \times \mathrm A_1 \subseteq \mathrm E_8$
  
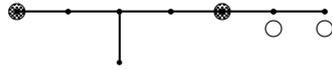
\[
  \begin{picture}(11400,2700)(-300,-1800)
    \multiput(0,0)(1800,0){6}{\usebox{\edge}}
    \put(3600,0){\usebox{\vedge}}
    \multiput(0,0)(7200,0){2}{\usebox{\gcircle}}
    \multiput(9000,0)(1800,0){2}{\usebox{\aprime}}
  \end{picture}
  \]

\item[27:] $\mathrm D_8 \subseteq \mathrm E_8$
  
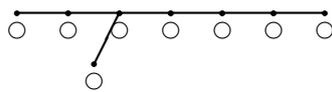
\[
  \begin{picture}(11400,3600)(-300,-2700)
    \put(0,0){\usebox{\dynkinaseven}}
    \thicklines
    \put(3600,0){\line(-1,-2){900}}
    \thinlines
    \put(2700,-1800){\circle*{150}}
    \multiput(0,0)(1800,0){7}{\usebox{\aprime}}
    \put(2700,-1800){\usebox{\aprime}}
  \end{picture}
  \]

\item[28:] $\mathrm B_4 \subseteq \mathrm F_4$
  \[
  \begin{picture}(5700,1800)(0,-900)
    \put(0,0){\usebox{\ffour}}
  \end{picture}
  \]
  
\item[29:] $\mathrm C_3 \times \mathrm A_1 \subseteq \mathrm F_4$
  \[
  \diagramfofour
  \]
  
\item[30:] $\mathrm A_1 \times \mathrm A_1 \subseteq \mathrm G_2$
  \[
  \begin{picture}(2400,1800)(-300,-900)
    \put(0,0){\usebox{\lefttriedge}}
    \multiput(0,0)(1800,0){2}{\usebox{\aprime}}
  \end{picture}
  \]
\end{description}

\clearpage

\section{Tables}
\label{sec:tables}
In this appendix, we have computed $\is{\Rs}$ for all reduced
elementary $\Rs$ with $\Ss$ from the list in
Appendix~\ref{sec:luna-diagr-symm} and $|\gdiv| = 1$.

\begin{table}[!ht]
  \begin{tabular}{@{}llll@{}}
    \toprule
    $R$ & $|R^+{\setminus}R^+_{S^p}|$ & $\|\Gamma\|$ & $\is{\tc{}}$ \\
    \midrule
    $A_n$, $n{\ge} 1$ & $n^2{+}n$ & $\gamma_k$, $\gamma_{n-k+1}$, $1 {\le} k {\le} \lceil \frac{n}{2} \rceil$ & $n^2 {-} 2kn {+}3n {+} 2k^2 {-} 6k {+} 4$\\
    $B_n$, $n{\ge} 2$ & $2n^2$ & $\gamma_1$ & $3n {-} 1$ \\
    && $\gamma_k$, $1 {<} k {<} n$ & $3n{+}k^2{-}2k{-}4 $\\
    && $\gamma_n$ & $n^2 {-} n$ \\
    $C_n$, $n {\ge} 2$ & $2n^2$ & $\gamma_k$, $1 {\le} k {<} n$ & $n {+} k^2 {-} 1$\\
    && $\gamma_n $ & $n^2 {+} 1$ \\
    $D_n$, $n {\ge} 4$ & $2n^2{-}2n$ & $\gamma_1$ & $3n{-}3$ \\
    && $\gamma_k$, $1 {<} k {\le} n-2$ & $3n{+}k^2{-}2k{-}6$ \\
    && $\gamma_{n-1}$, $\gamma_n$ & $n^2 {-} 2n {+} 1$ \\
    \bottomrule
  \end{tabular}
  \caption{classical group embeddings}
  \label{tab:cdiag}
\end{table}

\begin{table}[!ht]
  \begin{tabular}{@{}llll@{}}
    \toprule
    $R$ & $|R^+{\setminus}R^+_{S^p}|$ & $\|\Gamma\|$ & $\is{\tc{}}$\\
    \midrule
    $E_6$ & $72$ & $\gamma_1$, $\gamma_6$ & $\frac{37}{2}$\\
    && $\gamma_2$, $\gamma_3$, $\gamma_5$ & $16$\\
    && $\gamma_4$ & $14$\\
    $E_7$ & $126$ & $\gamma_1$ & $27$\\
    && $\gamma_2$, $\gamma_3$ & $25$\\
    && $\gamma_4$ & $22$\\
    && $\gamma_5$, $\gamma_6$ & $19$\\
    && $\gamma_7$ & $20$\\
    $E_8$ & $240$ & $\gamma_1$ & $38$ \\
    && $\gamma_2$, $\gamma_3$ & $36$ \\
    && $\gamma_4$ & $32$ \\
    && $\gamma_5$ & $27$ \\
    && $\gamma_6$ & $24$ \\
    && $\gamma_7$ & $22$ \\
    && $\gamma_8$ & $21$ \\
    $F_4$ & $48$ & $\gamma_1$ & $12$\\
    && $\gamma_2$ & $10$\\
    && $\gamma_3$ & $8$\\
    && $\gamma_4$ & $7$\\
    $G_2$ & $12$ & $\gamma_1$ & $2$ \\
    && $\gamma_2$ & $4$ \\
    \bottomrule
  \end{tabular}
  \caption{exceptional group embeddings}
  \label{tab:ediag}
\end{table}

\begin{table}[!ht]
  \renewcommand{\arraystretch}{1.1}
  \begin{tabular}{@{}llll@{}}
    \toprule
    No. & $|R^+{\setminus}R^+_{S^p}|$ & $\|\Gamma\|$ & $\is{\tc{}}$\\
    \midrule
    3 & $2m^2{+}2lm$ & $\gamma_1$ & $3m{+}2l{-}1$ \\
    & ${+}m{+}2l{-}1$ & $\gamma_k$, $1{<}k{\le}m$ & $3m{+}2l{+}k^2{-}2k{-}4$ \\
    && $\gamma_{m+1}$, $1{\le}m$ & $m^2{+}lm{+}l{-}2$ \\
    4 & $2m^2{+}3m{+}1$ & $\gamma_k$, $1{\le}k{\le}m$ & $m{+}k^2{-}1$ \\
    && $ \gamma_{m+1}$ & $m^2{+}m{+}1$\\
    \midrule
    5 & $\frac{1}{2}m^2{+}\frac{1}{2}m$ & $\gamma_k$, $\gamma_{m-k+1}$, $1{\le}k{\le}\lceil\frac{m}{2} \rceil$ & $\frac{1}{2}m^2{-}km{+}\frac{3}{2}m{+}k^2{-}4k{+}3$ \\
    \midrule
    6 & $2m^2{+}2m$ & $\gamma_k$, $\gamma_{m-k+1}$, $1{\le}k{\le}\lceil \frac{m}{2} \rceil$ & $2m^2{-}4km{+}6m{+}4k^2{-}10k{+}6$ \\
    \midrule
    8 & $4l$ & $\gamma_1$ & $2l{-}2$ \\
    && $\gamma_2$ & $4l{-}2$ \\
    9 & $m^2{+}2lm$ & $\gamma_1$ & $m{+}2l{-}1$ \\
    & ${+}2l{-}1$ & $\gamma_k$, $1{<}k{<}m$ & $m{+}2l{+}\frac{1}{2}k^2{-}\frac{1}{2}k{-}4$ \\
    && $\gamma_{m}$, $l{=}1$ & $\frac{1}{2}m^2{-}1$ \\
    && $\gamma_m$, $l{\ge}2$ & $\frac{1}{2}m^2{+}\frac{1}{2}m{+}2l{-}4$ \\
    && $\gamma_{m+1}$, $m {\ge} 1$ & $\frac{1}{2}m^2{+}lm{-}m{+}l{-}\tfrac{3}{2}$ \\
    \midrule
    10, 11 & $4m^2{+}4lm$ & $\gamma_k$, $1{\le}k{\le}m$ & $3m{+}2l{+}2k^2{-}k{-}1$ \\
    & ${+}3m{+}4l{-}1$ & $\gamma_{m+1}$, $l{=}1$ & $2m^2{+}4m{+}3$\\
    && $\gamma_{m+1}$, $l{\ge}2$ & $2m^2{+}2lm{+}2m{+}2l$\\
    \midrule
    12 & $m^2{+}2m{+}1$ & $\gamma_1$ & $m{+}1$ \\
    && $\gamma_k$, $1{<}k{\le}m$ & $m{+}\frac{1}{2}k^2{-}\frac{1}{2}k{-}2$ \\
    && $\gamma_{m+1}$ & $\frac{1}{2}m^2{+}\frac{1}{2}m{-}1$ \\
    13 & $m^2{+}2m{+}1$ & $\gamma_k$, $1 {\le} k {\le} m$ & $\frac{1}{2}k^2{+}\frac{1}{2}k{-}1$  \\
    && $\gamma_{m+1}$ & $\frac{1}{2}m^2{+}\frac{1}{2}m{+}1$ \\
    \midrule
    14 & $4l{+}2$ & $\gamma_1$ & $2l{-}1$ \\
    && $\gamma_2$ & $4l$ \\
    15 & $m^2{+}2lm$ & $\gamma_1$ & $m{+}2l$ \\
    & ${+}m{+}2l$ & $\gamma_k$, $1{<}k{<}m$ & $m{+}2l{+}\frac{1}{2}k^2{-}\frac{1}{2}k{-}3$ \\
    && $\gamma_m$, $l{\ge}1$ & $\frac{1}{2}m^2{+}\frac{1}{2}m{+}2l{-}3$ \\
    && $\gamma_{m+1}$, $l{\ge}1$, $m {\ge} 1$ & $\frac{1}{2}m^2{+}lm{-}\frac{1}{2}m{+}l{-}1$ \\
    && $\gamma_m$, $\gamma_{m+1}$, $l{=}0$, $m {\ge} 1$ & $\tfrac{1}{2}m^2{-}\tfrac{1}{2}m$ \\
    \midrule
    16/1 & $4m^2{+}8m{+}5$ & $\gamma_1$ & $7m{+}5$ \\
    && $\gamma_k$, $1{<}k{\le}m$ & $7m{+}2k^2{-}5k{+}2$ \\
    && $\gamma_{m+1}$ & $2m^2{+}4m{+}1$ \\
    16/2 & $4m^2{+}4m{+}1$ & $\gamma_1$ & $7m{+}1$ \\
    && $\gamma_k$, $1{<}k{\le}m$ & $7m{+}2k^2{-}5k{-}2$ \\
    && $\gamma_{m+1}$ & $2m^2{+}2m{-}1$ \\
    17 & $4m^2{+}4m{+}1$ & $\gamma_k$, $1{\le}k{\le}m$ & $3m{+}2k^2{-}k{-}1$ \\
    && $\gamma_{m+1}$ & $2m^2{+}2m{+}1$ \\
    \bottomrule
  \end{tabular}
  \caption{symmetric subgroups of classical simple groups}
  \label{tab:csg}
\end{table}

\begin{table}[!ht]
  \begin{tabular}{@{}llll@{}}
    \toprule
    No. & $|R^+{\setminus}R^+_{S^p}|$ & $\|\Gamma\|$ & $\is{\tc{}}$ \\
    \midrule
    18 & $30$ & $\alpha_1{+}\alpha_3{+}\alpha_4{+}\alpha_5{+}\alpha_6$ & $13$ \\
    && $2\alpha_2{+}2\alpha_4{+}\alpha_3{+}\alpha_5$ & $20$ \\
    \midrule
    19 & $24$ & $2\alpha_1{+}2\alpha_3{+}2\alpha_4{+}\alpha_2{+}\alpha_5$,
    $2\alpha_6{+}2\alpha_5{+}2\alpha_4{+}\alpha_2{+}\alpha_3$ & $24$ \\
    \midrule
    20 & $36$ & $\alpha_1{+}\alpha_6$ & $4$ \\
    && $\alpha_3{+}\alpha_5$ & $5$ \\
    && $2\alpha_4$ & $6$ \\
    && $2\alpha_2$ & $7$ \\
    \midrule
    21 & $36$ & $2\alpha_1$, $2\alpha_6$ & $\frac{13}{2}$ \\
    && $2\alpha_2$, $2\alpha_3$, $2\alpha_5$ & $5$ \\
    && $2\alpha_4$ & $4$ \\
    \midrule
    22 & $51$ & $2\alpha_1{+}2\alpha_3{+}2\alpha_4{+}\alpha_2{+}\alpha_5$ & $31$ \\
    && $2\alpha_6{+}2\alpha_5{+}2\alpha_4{+}\alpha_2{+}\alpha_3$ & $22$ \\
    && $\alpha_7$ & $23$ \\
    23 & $51$ & $2\alpha_1{+}2\alpha_3{+}2\alpha_4{+}\alpha_2{+}\alpha_5$ & $14$ \\
    && $2\alpha_6{+}2\alpha_5{+}2\alpha_4{+}\alpha_2{+}\alpha_3$ & $23$ \\
    && $2\alpha_7$ & $25$ \\
    \midrule
    24 & $60$ & $2\alpha_1$ & $13$ \\
    && $2\alpha_3$ & $12$ \\
    && $\alpha_2{+}2\alpha_4{+}\alpha_5$ & $11$ \\
    && $\alpha_5{+}2\alpha_6{+}\alpha_7$ & $9$ \\
    \midrule
    25 & $63$ & $2\alpha_1$ & $10$ \\
    && $2\alpha_2$, $2\alpha_3$ & $9$ \\
    && $2\alpha_4$ & $8$ \\
    && $2\alpha_5$, $2\alpha_6$ & $6$ \\
    && $2\alpha_7$ & $\frac{13}{2}$ \\
    \midrule
    26 & $104$ & $2\alpha_1{+}2\alpha_3{+}2\alpha_4{+}\alpha_2{+}\alpha_5$ & $19$ \\
    && $2\alpha_6{+}2\alpha_5{+}2\alpha_4{+}\alpha_2{+}\alpha_3$ & $23$ \\
    && $2\alpha_7$ & $24$ \\
    && $2\alpha_8$ & $25$ \\
    \midrule
    27 & $120$ & $2\alpha_1$ & $15$ \\
    && $2\alpha_2$, $2\alpha_3$ & $14$ \\
    && $2\alpha_4$ & $13$ \\
    && $2\alpha_5$ & $10$ \\
    && $2\alpha_6$ & $8$ \\
    && $2\alpha_7$ & $7$ \\
    && $2\alpha_8$ & $\frac{13}{2}$ \\
    \midrule
    28 & $15$ & $\alpha_1{+}2\alpha_2{+}3\alpha_3{+}2\alpha_4$ & $10$ \\
    \midrule
    29 & $24$ & $2\alpha_1$ & $4$ \\
    && $2\alpha_2$ & $3$ \\
    && $2\alpha_3$ & $2$ \\
    && $2\alpha_4$ & $\frac{3}{2}$ \\
    \midrule
    30 & $6$ & $2\alpha_1$ & $0$ \\
    && $2\alpha_2$ & $1$ \\
    \bottomrule
  \end{tabular}
  \caption{symmetric subgroups of exceptional simple groups}
  \label{tab:exc}
\end{table}

\clearpage

\section*{Acknowledgments}
The authors would like to thank Victor Batyrev, J\"urgen Hausen, and Simon
Keicher for several fruitful discussions.

\bibliographystyle{amsalpha}
\bibliography{gmsv}

\providecommand{\bysame}{\leavevmode\hbox to3em{\hrulefill}\thinspace}
\providecommand{\MR}{\relax\ifhmode\unskip\space\fi MR }
\providecommand{\MRhref}[2]{%
  \href{http://www.ams.org/mathscinet-getitem?mr=#1}{#2}
}
\providecommand{\href}[2]{#2}
\begin{thebibliography}{BCDD03}

\bibitem[ADHL15]{coxrings}
Ivan Arzhantsev, Ulrich Derenthal, J{\"u}rgen Hausen, and Antonio Laface,
  \emph{{Cox Rings}}, Cambridge Studies in Advanced Mathematics, vol. 144,
  Cambridge University Press, Cambridge, 2015.

\bibitem[Bat94]{Bat:DualPolyhedra}
Victor~V. Batyrev, \emph{Dual polyhedra and mirror symmetry for {C}alabi-{Y}au
  hypersurfaces in toric varieties}, J. Algebraic Geom. \textbf{3} (1994),
  no.~3, 493--535.

\bibitem[BCDD03]{bcdd}
Laurent Bonavero, Cinzia Casagrande, Olivier Debarre, and St{\'e}phane Druel,
  \emph{Sur une conjecture de {M}ukai}, Comment. Math. Helv. \textbf{78}
  (2003), no.~3, 601--626.

\bibitem[BL11]{f4}
P.~Bravi and D.~Luna, \emph{{An introduction to wonderful varieties with many
  examples of type {$\rm F_4$}}}, J. Algebra \textbf{329} (2011), 4--51.

\bibitem[Bou68]{Bourbaki:Lie}
N.~Bourbaki, \emph{\'{E}l\'ements de math\'ematique. {F}asc. {XXXIV}. {G}roupes
  et alg\`ebres de {L}ie. {C}hapitre {IV}: {G}roupes de {C}oxeter et syst\`emes
  de {T}its. {C}hapitre {V}: {G}roupes engendr\'es par des r\'eflexions.
  {C}hapitre {VI}: syst\`emes de racines}, Actualit\'es Scientifiques et
  Industrielles, No. 1337, Hermann, Paris, 1968.

\bibitem[BP14]{bp}
P.~Bravi and G.~Pezzini, \emph{Wonderful subgroups of reductive groups and
  spherical systems}, J. Algebra \textbf{409} (2014), 101--147.

\bibitem[BP15]{bp:sph-sys}
\bysame, \emph{The spherical systems of the wonderful reductive subgroups}, J.
  Lie Theory \textbf{25} (2015), no.~1, 105--123.

\bibitem[BR96]{sr-benrat}
Chal Benson and Gail Ratcliff, \emph{A classification of multiplicity free
  actions}, J. Algebra \textbf{181} (1996), no.~1, 152--186.

\bibitem[Bra13]{br13}
P.~Bravi, \emph{Primitive spherical systems}, Trans. Amer. Math. Soc.
  \textbf{365} (2013), no.~1, 361--407.

\bibitem[Bri90]{brg}
Michel Brion, \emph{Vers une g\'en\'eralisation des espaces sym\'etriques}, J.
  Algebra \textbf{134} (1990), no.~1, 115--143.

\bibitem[Bri93]{sphmori}
\bysame, \emph{Vari\'et\'es sph\'eriques et th\'eorie de {M}ori}, Duke Math. J.
  \textbf{72} (1993), no.~2, 369--404.

\bibitem[Bri97]{Brion:cc}
M.~Brion, \emph{{Curves and divisors in spherical varieties}}, {Algebraic
  groups and {L}ie groups}, {Austral. Math. Soc. Lect. Ser.}, vol.~9, Cambridge
  Univ. Press, Cambridge, 1997, pp.~21--34.

\bibitem[Bri07]{brcox}
Michel Brion, \emph{{The total coordinate ring of a wonderful variety}}, J.
  Algebra \textbf{313} (2007), no.~1, 61--99.

\bibitem[BVS15]{BS:Moduli}
Paolo Bravi and Bart Van~Steirteghem, \emph{The moduli scheme of affine
  spherical varieties with a free weight monoid}, Int. Math. Res. Not. IMRN,
  electronically published on October 5, 2015, DOI:
  http://dx.doi.org/10.1093/imrn/rnv281 (to appear in print).

\bibitem[Cam01]{camus}
Romain Camus, \emph{Vari\'{e}t\'{e}s sph\'{e}riques affines lisses}, Th\`{e}se
  de doctorat, Universit\'{e} Joseph Fourier, 2001.

\bibitem[Cas06]{cas06}
Cinzia Casagrande, \emph{The number of vertices of a {F}ano polytope}, Ann.
  Inst. Fourier (Grenoble) \textbf{56} (2006), no.~1, 121--130.

\bibitem[CF14]{cf2}
St\'ephanie Cupit-Foutou, \emph{Wonderful varieties: a geometrical
  realization}, arXiv:0907.2852v4.

\bibitem[Cox95]{tcox}
David~A. Cox, \emph{The homogeneous coordinate ring of a toric variety}, J.
  Algebraic Geom. \textbf{4} (1995), no.~1, 17--50.

\bibitem[DCP83]{cp83}
C.~De~Concini and C.~Procesi, \emph{Complete symmetric varieties}, Invariant
  theory ({M}ontecatini, 1982), Lecture Notes in Math., vol. 996, Springer,
  Berlin, 1983, pp.~1--44.

\bibitem[Gag14]{gag14}
Giuliano Gagliardi, \emph{The {C}ox ring of a spherical embedding}, J. Algebra
  \textbf{397} (2014), 548--569.

\bibitem[Gag15]{gcscsv}
\bysame, \emph{A combinatorial smoothness criterion for spherical varieties},
  Manuscripta Math. \textbf{146} (2015), no.~3-4, 445--461.

\bibitem[GH15a]{gsfv}
Giuliano Gagliardi and Johannes Hofscheier, \emph{{G}orenstein spherical {F}ano
  varieties}, Geom. Dedicata, electronically published on February 11, 2015,
  DOI: http://dx.doi.org/10.1007/s10711-015-0047-y (to appear in print).

\bibitem[GH15b]{gh15}
\bysame, \emph{Homogeneous spherical data of orbits in spherical embeddings},
  Transform. Groups \textbf{20} (2015), no.~1, 83--98.

\bibitem[Kno91]{knopsph}
Friedrich Knop, \emph{{The {L}una-{V}ust theory of spherical embeddings}},
  {Proceedings of the {H}yderabad {C}onference on {A}lgebraic {G}roups
  ({H}yderabad, 1989)} (Madras), Manoj Prakashan, 1991, pp.~225--249.

\bibitem[KP85]{kp85}
Hanspeter Kraft and Vladimir~L. Popov, \emph{Semisimple group actions on the
  three-dimensional affine space are linear}, Comment. Math. Helv. \textbf{60}
  (1985), no.~3, 466--479.

\bibitem[Lea98]{sr-leahy}
Andrew~S. Leahy, \emph{A classification of multiplicity free representations},
  J. Lie Theory \textbf{8} (1998), no.~2, 367--391.

\bibitem[Los09]{losev-uniq}
Ivan~V. Losev, \emph{Uniqueness property for spherical homogeneous spaces},
  Duke Math. J. \textbf{147} (2009), no.~2, 315--343.

\bibitem[Lun97]{Luna:cc}
D.~Luna, \emph{{Grosses cellules pour les vari{\'e}t{\'e}s sph{\'e}riques}},
  {Algebraic groups and {L}ie groups}, {Austral. Math. Soc. Lect. Ser.},
  vol.~9, Cambridge Univ. Press, Cambridge, 1997, pp.~267--280.

\bibitem[Lun01]{Luna:typea}
\bysame, \emph{{Vari{\'e}t{\'e}s sph{\'e}riques de type {$A$}}}, Publ. Math.
  Inst. Hautes {\'E}tudes Sci. (2001), no.~94, 161--226.

\bibitem[LV83]{lunavust}
D.~Luna and Th. Vust, \emph{{Plongements d'espaces homog{\`e}nes}}, Comment.
  Math. Helv. \textbf{58} (1983), no.~2, 186--245.

\bibitem[Pas08]{Pasquier:FanoHorospherical}
Boris Pasquier, \emph{{Vari{\'e}t{\'e}s horosph{\'e}riques de {F}ano}}, Bull.
  Soc. Math. France \textbf{136} (2008), no.~2, 195--225.

\bibitem[Pas10]{pas10}
\bysame, \emph{The pseudo-index of horospherical {F}ano varieties}, Internat.
  J. Math. \textbf{21} (2010), no.~9, 1147--1156.

\bibitem[Per14]{per14}
Nicolas Perrin, \emph{On the geometry of spherical varieties}, Transform.
  Groups \textbf{19} (2014), no.~1, 171--223.

\bibitem[Ste68]{st68}
Robert Steinberg, \emph{Endomorphisms of linear algebraic groups}, Memoirs of
  the American Mathematical Society, No. 80, American Mathematical Society,
  Providence, R.I., 1968.

\bibitem[Tim11]{ti}
Dmitry~A. Timashev, \emph{{Homogeneous spaces and equivariant embeddings}},
  {Encyclopaedia of Mathematical Sciences}, vol. 138, Springer, Heidelberg,
  2011, Invariant Theory and Algebraic Transformation Groups, 8.

\bibitem[Vus90]{vust90}
Thierry Vust, \emph{Plongements d'espaces sym\'etriques alg\'ebriques: une
  classification}, Ann. Scuola Norm. Sup. Pisa Cl. Sci. (4) \textbf{17} (1990),
  no.~2, 165--195.

\end{thebibliography}

\end{document}